\documentclass[11pt,leqno]{amsart}
\usepackage{amsmath}

\usepackage{amsfonts}
\usepackage{mathtools}
\usepackage{amssymb}
\usepackage{mathrsfs}  
\usepackage{bookmark} 

\usepackage[a4paper, centering]{geometry}
\usepackage{epstopdf}
\usepackage{amscd}
\usepackage[dvipsnames]{xcolor}
\usepackage{graphicx}
\usepackage{subcaption}  
\usepackage{yfonts}

\newtheorem{teo}{Theorem}[section]
\newtheorem{lem}[teo]{Lemma}
\newtheorem{cor}[teo]{Corollary}
\newtheorem{prop}[teo]{Proposition}
\newtheorem{es}[teo]{Examples}

\newcommand{\medint}{-\kern  -,395cm\int}

\theoremstyle{remark}
\newtheorem{oss}[teo]{Remark}

\theoremstyle{definition}
\newtheorem{defi}[teo]{Definition}

\theoremstyle{definition}

\newcommand{\norm}[1]{\left\|{#1}\right\|}
\newcommand{\dive}{\mathrm{div}}

\newcommand{\supp}{\mathrm{supp}}

\newcommand{\bbR}{{\mathbb{R}}}
\newcommand{\bbN}{{\mathbb{N}}}

\usepackage{mathtools}
\usepackage{esint}

\newcommand{\R}{\bbR}

\newcommand{\MM}{\mathcal{M}}

\newcommand{\LL}{\mathcal{L}}

\def\nm #1{ \left\langle #1 \right\rangle }

\def\de{{\rm d}}
\def\BV{{\rm BV}}

\def\PR{{\rm Per}}

\def\cD{{\mathcal D}}
\def\A{{w}}
\newcommand{\astar}{\hat{{w}}}

\newcommand\weak{{\rightharpoonup}}

\newcommand{\norma}[1]{{\left\| #1\right\|}}

\def\upiu{u^+}
\def\umeno{u^-}
\def\uint{{u^{i}}}
\def\uext{{u^{e}}}

\newcommand{\Haus}[1]{{\mathscr H}^{#1}} 

\makeatletter
\def\cleardoublepage{\clearpage\if@twoside \ifodd\c@page\else
\hbox{}
\thispagestyle{empty}
\newpage
\if@twocolumn\hbox{}\newpage\fi\fi\fi}
\makeatother
\title{Relaxation for a degenerate functional with linear growth in the onedimensional case}

\author[V.~Chiad\`o Piat]{Valeria Chiad\`o Piat}
\address{Dipartimento di scienze Matematiche Giuseppe Lagrange, Corso duca degli abruzzi, 24, Torino (Italy)}
\email{valeria.chiadopiat@polito.it}\author[V.~De Cicco]{Virginia De Cicco}
\address{Dipartimento di Scienze di Base  e Applicate per l'Ingegneria, Sapienza Univ.\ di Roma\\
	Via A.\ Scarpa 10 -- I-00185 Roma (Italy)}
\email{virginia.decicco@uniroma1.it}
\author[A.~Melchor Hernandez]{Anderson Melchor Hernandez}
\address{Dipartimento di Matematica, Via Zamboni, 33, 40126, Bologna (Italy)}
\email{anderson.melchor@unibo.it}


\begin{document}

\begin{abstract}
In this work, we study the relaxation of a degenerate functional with linear growth, depending on a weight $w$ that does not exhibit doubling or Muckenhoupt-type conditions.
In order to obtain an explicit representation of the relaxed functional and its domain, our main tools for are Sobolev inequalities with double weight.
\end{abstract}
\maketitle

\tableofcontents
\section{Introduction}
In this work, we focus on the study of an integral functional in one dimension with linear growth, allowing for a degenerate weight $w$.  We aim to provide an explicit relaxation formula for the functional 
\begin{align}\label{FunctF}
F(u)\coloneqq\left\{
\begin{aligned}
& \int_{\Omega} |u'|w \de x\text{ \ \ if } u\in  {\mathrm AC}(\overline\Omega),\\
& +\infty \ \ \ \ \ \ \ \ \ \ \ \ \ \ \text{if}\hskip 0,1cm  u\in X\setminus {\mathrm AC}(\overline\Omega),
\end{aligned}
\right.
\end{align}
where $\Omega$ is an open bounded set in $\mathbb{R}$, $u'$ denotes the derivative of $u$, $w$ is a nonnegative, locally integrable function, ${\rm AC}(\overline\Omega)$ is the space of absolutely continuous functions on $\overline{\Omega}$, and $X$ is a topological space comprising measurable functions which will be introduced later on. 
We will find an explicit expression of the lower semicontinuous envelope of $F$, that is denoted by $\overline{F}$ with respect to a suitable convergence. \\
Several studies have focused on investigating functionals with $p$-growth for $1<p<+\infty$ within different functional frameworks; see, for example, \cite{CD,CC,FM,
Ma}.
Nevertheless, there are few works dedicated to the analysis of functionals with linear growth like \eqref{FunctF} above (see, for instance, \cite{Braides1995} and references therein).
In the recent work \cite{CCMH}, we have analyzed the $p$-version of the functional $F$, defined as
\begin{align}\label{FunctFp}
F_{p}(u)\coloneqq\left\{
\begin{aligned}
& \int_{\Omega} |u'|^{p}w\, \de x\text{ \ \ if } u\in {\mathrm AC}(\overline\Omega),\\
& +\infty \ \ \ \ \ \ \ \ \ \ \ \ \ \ \text{if}\hskip 0,1cm  u\in X\setminus {\mathrm AC}(\overline\Omega),
\end{aligned}
\right.
\end{align}
where $w$ does not exhibit doubling or Muckenhoupt-type conditions, \cite{Muk}.
In that case, we have conducted the analysis in weighted Sobolev spaces; we refer to \cite{LAMB2,LAMB3,LAMB1,HKST} for general approaches to the definition of these spaces.
Let us briefly explain our strategy in the case $1<p<+\infty$, and what is different in the present case $p=1$. We first proved Poincar\'e inequalities involving
$w$ and an auxiliary weight $\hat{w}_{p}$ that corrects the weight in the zones where $w$ is  strongly degenerate (i.e. $w^{-\frac{1}{p-1}}$ is not summable). Specifically, we showed that the
$p$-norm of the gradient term of a generic function $u$ weighted by $w$,  is greater up to a suitable constant than the $p$-norm of $u$ weighted by $(\hat{w}_{p})^{p-1}$. 
Subsequently, assuming that $w$ is finitely degenerate (see \cite[Definition 2.1]{CCMH}), and in view of such a Poincaré inequality with two different weights, we proceeded to choose $X=L^{p}((\hat{w}_{p})^{p-1})$, and showed that $AC$-functions are dense, in a suitable Sobolev space $W\subseteq X$. As a consequence, we were able to determine the finiteness domain of the relaxed functional $\overline{F}_{p}$ by performing the relaxation in the strong topology of $X$. 

In the present work, we follow some of the previous ideas, but we cannot apply verbatim such methodology. 
The first reason is that for a functional with linear growth like \eqref{FunctF}, it is necessary to work with $\BV$ like spaces, rather than Sobolev spaces, and the second reason is that the functional in this case can be interpreted as a pairing.

A class of weighted bounded variation functions ${\rm BV}(\Omega;w)$ in any dimension ($\Omega\subset\R^n$) is introduced in \cite{Baldi} (see Section \ref{baldibv} where we recall the definitions and the results of \cite{Baldi}).\\ 
By requiring that $\A>0$ and $\A$ belongs to the Muckenhoupt class $A_{1}$, (see Definition \ref{first:A1} below) it is possible to define a weighted ${\rm BV}(\Omega,w)$-space. A priori such weight $\A$ is only a.e. defined, but it is not restrictive to assume that condition $A_{1}$ holds for any point in $\Omega$ (this is possible since it can be proved that there exists a further weight lower semicontinuous $\tilde{w}$ that defines the same weighted ${\rm BV}$-space, and satisfies $A_{1}$ at any point, see Lemma \ref{equivalentbv} below). Moreover a density theorem holds true in ${\rm BV}(\Omega,w)$ (see Theorem \ref{densbv} below) and by assuming the local growth  condition \eqref{localgrowth} a Poincar\'e inequality holds (see Theorem \ref{baldipoincarepesata} below).

In the present paper, although confining the study to the onedimensional case, we follow another approach.
We will deal with a weight $w\geq 0$ (and so it admits large degeneration), that does not belong to the Muckenhoupt class $A_{1}$ (and so it is only a.e. defined) and does not satisfy any doubling condition.
We will consider a new category of spaces that we denote as $\BV^{\A}_{{\rm loc}}(\Omega)$ inspired to some $BV$ like space recently introduced in \cite{COM}, although this approach forces us to assume some regularity of the weight, i.e. $w$ is a $BV_{{\rm loc}}$ within the largest open set where $\frac{1}{\A}$ is bounded. 

More precisely, we say that $u\in \BV^{\A}_{{\rm loc}}(\Omega)$ if it is a Borel function that belongs to $L^1_{\rm{loc}}(\Omega, w)\cap L^1_{\rm{loc}}(\Omega, |Dw|)$, such that the Anzellotti pairing $(\A,Du)$, defined below
is a Radon measure (see \cite{ANZ} for its original definition). 
Morover, under  suitable assumptions this class  is a Banach space. 

Under the assumption $w\in BV_{\rm{loc}}(\Omega)$, the distributional definition of the pairing is the following \begin{align}\label{introduction:pairing}
\nm{(\A,Du),\varphi}\coloneqq -\int_{\Omega}u^{\frac12} \varphi\,\de D\A- \int_{\Omega} u  \varphi^{\prime}\A\,\de x, \hskip 0,1cm \text{\ \ for $\varphi\in C_{c}^{\infty}(\Omega)$.}
\end{align}

Here $\varphi^\prime$ denotes the derivative of $\varphi$, $Du$ denotes the distributional derivative of $u$, and $u^{\frac12}$ the precise representative of $u$ (see Appendix 1 for a more detail explanation). The space $\BV^{\A}(\Omega)$ was introduced in \cite{COM} because it is the natural functional space where the distributional derivative defined in \eqref{introduction:pairing} is a Radon measure. 
\\
In the present work, we find that  $\BV^{\A}(\Omega)$ is the natural ambient space in which an explicit formula for $\overline{F}$ can be expressed. Therefore, to ensure a suitable behavior of \eqref{introduction:pairing}, we restrict our analysis to the following setup. We assume that $\A$ is a nonnegative function such that $\A$ is locally integrable in $\Omega$.
Our objective is to demonstrate that, under these conditions, the relaxed functional can be expressed by means of  a pairing, as studied in \cite{COM} and \cite{COM1}. 
This pursuit is built upon innovative concepts introduced in those works, where $\BV^{\A}(\Omega)$ spaces, 
consisting of functions that satisfy divergence-measure properties, are larger than the conventional $BV(\Omega)$-spaces in \cite{AFP}, or the weighted $BV(\Omega,\A)$-spaces in \cite{Baldi}.
By following \cite{CC,CCMH}, our chosen space $X$ comprises $W^{1,1}$-functions with a degenerate weight $\A$. 
The pairing of such functions $u$ with $\A$ consists in a Radon measure within the largest open set where $\frac{1}{\A}$ is bounded. 
This requires the introduction of an additional weight, denoted as $\astar$. This corrective function addresses the singularities inherent in the respective weight $\A$. 
Moreover, in this scenario, we also prove a weighted Poincar\'e inequality involving $\A$ and $\astar$. 

Subsequently, in Section 3, we assume that $w$ is finitely degenerate (see Definition 2.1 below)
and the stronger condition that the weight $w$  belongs to $W_{{\rm loc}}^{1,1}$ within the largest open set where $\frac{1}{\A}$ is bounded.
We then relax ${F}$ with respect to a weak convergence involving $\astar$ and $\vert D\A\vert$, which we will refer to as $(\astar,D\A)$-convergence.
This is similar to the $(\A,\frac{1}{2})$-convergence introduced in \cite{COM} (see Definition \ref{atwo} before). 
The main difference lies in the choice of the $L^{1}(\astar)$-weak convergence rather than $L^{1}(\A)$-weak convergence.\\

This work is structured as follows. In Section \ref{sec:4}, we define $\hat w$ and prove
the validity of weighted Poincar\'e inequalities, see Theorem \ref{poincare1} below. 
Thanks to this result, we are allowed in Section \ref{sec:relax1}, once introduced our $(\astar,D\A)$-convergence, to prove a compactness theorem with respect to this convergence and to prove our relaxation theorem,
see Definition \ref{ourconv}, and Theorem \ref{maindegen}, respectively. Lastly, in Appendix \ref{sec:append}, we revisit some fundamental concepts from geometric measure theory, applicable to all dimensions $n\geq 1$, 
and we recall the notion of pairing as studied in \cite{COM}.
In Section \ref{baldibv}, we recall some similar results about weighted Poincaré inequalities when $\A$ belongs to the Muckenhoupt class $A_{1}$, obtained in \cite{Baldi}.

\section{Poincar\'e inequalities with double weight}\label{sec:4}
Let $\Omega=(a,b)$ be a bounded open interval. In what follows, we make the following structural assumptions:

\begin{itemize}
\item[(H1)] $\A\geq 0$;
\item[(H2)]$\A\in L_{{\rm loc}}^{1}(\Omega)$.
\end{itemize}  Here we denote by $I_{\Omega,\A}$ the biggest open bounded set contained in $\Omega$ such that $\frac{1}{\A}$ is $L_{\rm loc}^{\infty}(I_{\Omega,\A})$-function. Then $I_{\Omega,\A}$  can be written in a unique way as the union of pairwise disjoint open intervals $(a_{i},b_{i})\subset \Omega$, that is,

\begin{align*}
I_{\Omega,\A}=\bigcup_{i=1}^{N_{\A}}(a_{i},b_{i}),
\end{align*}
with $1\leq N_{\A}\leq +\infty$. Furthermore, since $\frac{1}{\A}\in L_{\rm loc}^{\infty}(I_{\Omega,\A})$,  for every $i=1,\ldots,N_{\A}$ and $K\Subset (a_i,b_i)$ there exists a nonnegative constant $c_{i,K}$ such that 
\begin{align}\label{localc}
\frac1{\A(x)}\leq c_{i,K} \hskip 0,1cm \text{for a.e. $x\in K$.}
\end{align}
\begin{defi}
\begin{itemize}
\item[(i)] If $I_{\Omega,\A}=\,\emptyset$, we put $N_{\A}:=0$.
\item[(ii)]  If $1\le\,N_{\A}<\,\infty$  we say that $\A$ is {\it finitely degenerate} in $\Omega$. 
\item[(iii)]   If $N_{\A}=\,+\infty$  we say that $\A$ is {\it not finitely degenerate} in $\Omega$.
\end{itemize}
\end{defi}
\begin{es}\label{valeria} 
{\rm Let us consider the following examples.
\begin{enumerate}
\item[(I)] Let  $\A(x)=(1-x^2)^2$ defined in the interval  $(-2,2)$: then, $I_{\Omega,\A}=(-2,-1)\cup(-1,1)\cup(1,2)$, and $\A$ is finitely degenerate with $N_{\A}=3$. 
\item[(II)] Let  $\A(x)=1+\sin\frac1 x$ defined in the interval $(0,1)$: since $\A(x_i)=0$ if $x_i=\frac1{\pi(\frac32+2i)}$, $i\in\bbN$, we have that $I_{\Omega,\A}=\bigcup_{i\in\bbN}(x_{i+1},x_i)$ and $\A$ is not finitely degenerate, i.e. $N_{\A}=+\infty$.
\end{enumerate}
}
\end{es}

\subsection{An auxiliary weight}
Let $\astar:\Omega\to [0,+\infty[$ be defined as
\begin{equation*}
\astar(x):=
\begin{cases}
\displaystyle\lim_{x\to a_i^+}  \left(\norm{ \A^{-1}}_{L^{\infty}\left(\left(x,\frac{a_{i}+b_{i}}{2}\right)\right)}\right)^{-1}    &\text{ if }x=a_i \\
\left(\norm{ \A^{-1}}_{L^{\infty}\left(\left(x,\frac{a_{i}+b_{i}}{2}\right)\right)}\right)^{-1}  &\text{ if } {a_i}< x\leq {{3a_i+b_i}\over{4}}\\
\left(\norm{ \A^{-1}}_{L^{\infty}\left(\left(\frac{3a_{i}+b_{i}}{4},\frac{a_{i}+3b_{i}}{4}\right)\right)}\right)^{-1}  &\text{ if } {{3a_i+b_i}\over{4}}\leq x\leq {{a_i+3b_i}\over{4}}\\
 \left(\norm{ \A^{-1}}_{L^{\infty}\left(\left(\frac{a_{i}+b_{i}}{2},x\right)\right)}\right)^{-1}  &\text{ if } {{a_i+3b_i}\over{4}}\leq x<b_i\\
\displaystyle\lim_{x\to b_i^-} \left(\norm{ \A^{-1}}_{L^{\infty}\left(\left(\frac{a_{i}+b_{i}}{2},x\right)\right)}\right)^{-1}   &\text{ if } x=b_i\\
\ \ \ \ \ \ \ \ \ 0&\text{ if } x\in \Omega\setminus \overline I_{\Omega,\A}\,.
\end{cases}
\end{equation*}
\begin{oss}
At first glance, the definition of $\astar$ may seem subtle. Nevertheless, it is an important function with nice regularity properties, as presented in the next proposition, 
and it allows us to prove the validity of a Poincaré inequality with weights $\A$ and $\astar$, respectively.
It is also worth noting that a similar definition of the function $\astar$ was already considered in \cite{CCMH}, in the case where the functional $F$ defined in \eqref{FunctF} is replaced by \eqref{FunctFp}. 
Instead, the present work addresses the case $p=1$ separately, because the tools used in \cite{CCMH} were developed in a Sobolev context,
whereas here we need tools beyond ${\rm BV}(\Omega)$-spaces recently developed in \cite{COM1,COM,CRA,CRA1}.
\end{oss}
In the following figures, we illustrate the behavior of the function $\astar$ for a specific choice of $\A$, while in Proposition \ref{dualpropw}, we prove some of its mathematical properties.
\begin{figure}[!ht]
\centering
\begin{subfigure}[b]{0.4\textwidth}
\includegraphics[width=\textwidth]{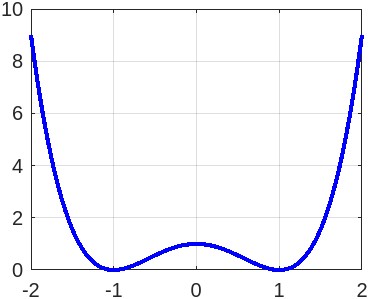}
\end{subfigure}
\hskip 0,3cm
  \begin{subfigure}[b]{0.4\textwidth}
    \includegraphics[width=\textwidth]{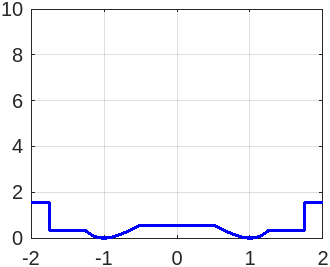}
  \end{subfigure}
   \caption{In the first figure on the left hand side, we have the profile of $w(x)=(1-x^{2})^{2}$ for  $x\in (-2,2)$, while in the right hand side, we have its associated weight $\astar$. In this case, we note that $N_{w}=3$.}
\end{figure}
 \vskip 0,1cm
Let us collect some properties of the function $\astar$ in the following Proposition.
\begin{prop}[Properties of $\astar$]\label{dualpropw}
Suppose that {\rm{(H1)-(H2)}} hold true.
\begin{itemize}
\item[(i)] For each $i=1,\dots, N_{\A}$, $\astar$ is constant in $[{{3a_i+b_i}\over{4}},{{a_i+3b_i}\over{4}}]$, increasing in $[{a_i},{{3a_i+b_i}\over{4}}]$ and so a ${\rm BV}$-function in $\left[a_{i},\frac{3a_{i}+b_{i}}{4}\right)$, decreasing in $[{{a_i+3b_i}\over{4}},b_i]$ and so a ${\rm BV}$-function in $\left(\frac{a_{i}+3b_{i}}{4},b_{i}\right]$. Moreover, it holds that
\begin{align}\label{stima1}
&0\,<\astar(x)\le\,\sup_{y\in (a_i,b_i)}\astar(y) \eqqcolon L_{i}<\,\infty\quad\forall\,x\in (a_i,b_i)\,,\\\label{stima2}
&M_{i,K}\coloneqq \inf_{x\in K}\astar(x)>\,0\text{ for each $x\in K\Subset (a_i,b_i)$,}
\end{align}
and $\astar(a_i)=\,0$ (respectively $\astar(b_i)=\,0$) if and only if $\frac{1}{ \A}\notin L^{\infty}((a_i,\frac{a_i+b_i}{2}))$ (respectively $\frac{1}{ \A}\notin L^{\infty}((\frac{a_i+b_i}{2},b_i))$). 
\item[(ii)] If $\frac{1}{ \A}\in  L^{\infty}(\Omega)$, then there exists a  constant $c>\,0$ such that
\[
0<\,\frac{1}{c}\le\,\astar(x)\le\,c\quad\text{ a.e. } x\in\Omega\,.
\]
\item[(iii)] If $\A$ is finitely degenerate in $\Omega$, i.e.  $1\le\,N_\A<\,\infty$, then there exists a  constant $c>\,0$ such that
\[
0\le\,\astar(x)\le\,c\quad\text{ a.e. }x\in\Omega\,
\]
and $\astar\in {\rm BV}(\Omega)$.
\item[(iv)] If $\A$ is not finitely degenerate in $\Omega$, i.e.  $N_{\A}=\,\infty$, then $\astar\in L^\infty_{\rm loc}(I_{\Omega,\A})$, and for each $1\leq i<\,+\infty$, we get $\astar\in {\rm BV}((a_{i},b_{i}))$.
\end{itemize}
\end{prop}

\begin{oss} By definition $\hat w\leq w$ and
fixed $i=1,\dots,N_{\A}$ if the function $\A$ is increasing in $(a_i,\frac{3a_i+b_i}{4})$, then $\astar(x)=\A(x)$ a.e. in $(a_i,\frac{3a_i+b_i}{4})$. This is the case in Examples \ref{valeria}. In the case (I), the function $\A$ is increasing in $(-1,-\frac12)$ and in $(1,\frac54)$, while in the case (II) the function $\A$ is  increasing in $(x_{i+1},\frac{3x_{i+1}+x_i}{4})$.
\\
On the contrary, 
if $\A$ admits an oscillating behaviour in a right neighborhood of some $a_i$, it can be happen that $\astar\neq\A$ in this neighborhood (see Example in Remark \ref{controes} below). \\
On the other hand, let us notice that, unlike the case $1<p<+\infty$, our weight $\astar$  involves the inverse of the $L^{\infty}$-norm of $\A^{-1}$. However, we can say that $\astar$ is a ${\rm BV}_{{\rm loc}}$ function rather than absolutely continuous as it happens in the case $1<p<+\infty$, \cite[see Proposition 2.5 (ii)]{CCMH}. It is important to recognize that, in some sense, the conditions assumed in Proposition \ref{dualpropw} are the analogue counterpart of those assumed in \cite[Proposition 2.5 (ii)]{CCMH}. Specifically, while in such a Proposition, we required hypotheses to give a meaning to the integral of  $w^{-\frac{1}{p-1}}$ for $1<p<+\infty$, Proposition \ref{dualpropw} involves the $L^{\infty}$-norm of $w^{-1}$.
\end{oss}

In what follows, given $\A\in {\rm BV}_{\rm loc}(I_{\Omega,\A})$ we set
\begin{equation}\label{bvspa}
{\rm{Dom}}_{\A}\coloneqq \Big\{u:\Omega\to\R: u\in W^{1,1}_{\text{loc}}(I_{\Omega, \A}),
u\in {\rm BV}_{\rm loc}^{\A}(I_{\Omega,\A})\Big\}\,,
\end{equation}
where the class  ${\rm BV}_{\rm loc}^{\A}(I_{\Omega,\A})$ has been defined  in the Introduction.
We note that this definition of ${\rm{Dom}}_{\A}$ differs from the one in \cite[formula 3]{CCMH}. Indeed, in \cite[formula (3)]{CCMH}, the definition of ${\rm{Dom}}_{w}$ does not require any regularity properties on the weight $w$. In fact, we have that in the case $1<p<+\infty$, ${\rm{Dom}}_{\A}$ is defined as

\begin{equation}\label{domwp}
{\rm{Dom}}_{w,p}:=\Big\{u:\Omega\to\R: u\in W^{1,1}_{\text{loc}}(I_{\Omega, w}),
\int_{I_{\Omega, w}} |u' |^p\,w\,\de x<+\infty\Big\}\,.
\end{equation}
The importance of the functional spaces \eqref{bvspa} and \eqref{domwp} is related to the relaxation result in Section \ref{sec:relax1} below and in \cite[Section 3]{CCMH}, respectively. For this reason we need to study the Poincar\'e inequality in ${\rm{Dom}}_{\A}$. 
\begin{oss}
The space ${\rm BV}_{\rm loc}^{\A}(I_{\Omega,\A})$ considered in the definition of ${\rm{Dom}}_{\A}$ in \eqref{bvspa} has been introduced recently in \cite{COM} in the general multidimensional setting.
We recall the definition and the main properties of ${\rm BV}_{\rm loc}^{\A}$ in the Appendix \ref{sec:append}, with the details in the onedimensional case. 
We notice that
$$
{\rm{Dom}}_{\A}\subset L^{1}_{\rm loc}(I_{\Omega,\A}, \A)\cap L^{1}_{\rm loc}(I_{\Omega,\A}, \vert {\rm D}\A\vert),
$$
and by the definition of pairing in \eqref{def:pairing} below
\begin{align*}
\nm{(\A,Du),\varphi}\coloneqq -\int_{I_{\Omega,\A}}u \varphi\,\de {\rm D}\A- \int_{I_{\Omega,\A}} u\varphi^{\prime}\A\,\de x, \hskip 0,2cm \text{for $\varphi\in C_{c}^{\infty}(I_{\Omega,\A}),\, u\in {\rm{Dom}}_{\A}$ }.
\end{align*}
Here we used that, since $u\in W^{1,1}_{\text{loc}}(I_{\Omega, \A})$ we have $u=u^{\ast}$ (recall that $u^\ast$ is the precise representative of $u$, and since in the onedimensional case $W^{1,1}(I)=AC(\overline{I})$, we have that $u=u^{\ast}=u^{\frac{1}{2}}$ where $u^{\frac{1}{2}}$ is the trace of $u$ as defined in \eqref{traccia12}), the measure $(\A,Du)$ has the following expression
\begin{align*}
(\A,Du)(I)=\int_{I}u^{\prime}(x)\A(x)\de x,
 \text{\ \ for any $I\Subset I_{\Omega,\A}$},
\end{align*}
and, by definition of ${\rm BV}_{\rm loc}^{\A}(I_{\Omega,\A})$, its total variation is finite
\begin{align*}
|(\A,Du)|(I)<+\infty.
\end{align*}
Let us note that we have used the symbol $u^{\prime}$ to denote the derivative of $u$. In what follows, we will maintain this notation and will subsequently use $Du$ to denote
the distributional derivative of $u$. Let us now give some comments about  the ambient space ${\rm BV}_{\rm loc}^{\A}$, and further weighted Sobolev spaces used in the literature.

\begin{enumerate}
\item[$\bullet$] Note that when $\A$ is lower semicontinuous, and belongs to the Muckenhoupt class $A_{1}$  in $\Omega$,  it is possible to define the weighted space $\BV_{{\rm loc}}(\Omega,\A)$ that consists of functions $u\in L^{1}(I;\A)$ such that $\int_{I}\A \de \vert D u\vert<+\infty$, for each $I\Subset \Omega$, see Section \ref{baldibv} below. 
\item[$\bullet$] Notice that ${\rm BV}_{\rm loc}^{\A}(\Omega)$ is defined by means of the Anzellotti pairing, whose definition requires the BV regularity of $\A$. Hence,  $\BV_{\rm loc}(\Omega,\A)$ and ${\rm BV}_{\rm loc}^{\A}(\Omega)$ share similar properties, however, they are different spaces, as we will explain, not only by construction. Let us recall that by \cite[Remark 5]{Baldi}, one has that $\BV(\Omega,\A)\subseteq \BV(\Omega)$ (and also $\BV_{\rm loc}(\Omega,\A)\subseteq \BV_{{\rm loc}}(\Omega)$ ). A major difficulty in the definition of $\BV(\Omega;\A)$  is that we need the Muckenhoupt class $A_{1}$ to hold at any point, rather than almost everywhere. 
\item[$\bullet$]Since in our context we do not assume that $\A$ belongs to the Muckenhoupt class $A_{1}$, a priori, we have that $\BV_{\rm loc}^{\A}(\Omega)$ and $\BV_{\rm loc}(\Omega,\A)$ are not comparable. However, we may wonder whether $\BV_{\rm loc}^{\A}(\Omega)$ and $\BV_{\rm loc}(\Omega, \A)$ are related (or if $\BV^{\A}(\Omega)$ and $\BV(\Omega, \A)$ are related). For the sake of a lean explanation, let us suppose that $\A$ is lower semicontinuous, and belongs to the Muckenhoupt class $A_{1}$  in $\Omega$, and that $\A\in L^{\infty}(\Omega)$.  Then by Remark \ref{contained1} below, we have that $\BV_{{\rm loc}}(\Omega)\subset \BV_{{\rm loc}}^{\A}(\Omega)$, and thus

\begin{align}\label{importantrel}
\BV_{{\rm loc}}(\Omega,\A)\subset \BV_{{\rm loc}}(\Omega)\subset \BV_{{\rm loc}}^{\A}(\Omega).
\end{align}
\end{enumerate}
\end{oss}

\begin{oss}\label{controes}
Next, we show that $L^{1}(\Omega,\astar)$ is generally not contained in $L^{1}(\Omega,\A)$. That is, we give an example of $\A$ such that there exists $u\in L^{1}(\Omega;\astar)$, $|(\A,Du)|(\Omega)=\int_{\Omega}\vert u' \A\vert \de x<+\infty$, but $u\notin L^{1}(\Omega,\A)$ and so $u\notin {\rm BV}^{\A}(\Omega)$. Let us set $\Omega\coloneqq (0,2)$, and for each $h\in \bbN$, $h\neq 0$ define

\begin{align*}
&I_{h}^{1}\coloneqq \left(\frac{1}{h+1}, \frac{1}{2}\left(\frac{1}{h+1}+\frac{1}{h}\right)\right]; \hskip 0,1cm  I_{h}^{2}\coloneqq \left(\frac{1}{2}\left(\frac{1}{h+1}+\frac{1}{h}\right),\frac{1}{h}\right],\\
&I^{1}\coloneqq \cup_{h=1}^{\infty}I_{h}^{1}; \hskip 0,2cm I^{2}\coloneqq \cup_{h=1}^{\infty}I_{h}^{2}; \hskip 0,2cm I_{h}\coloneqq I_{h}^{1}\cup I_{h}^{2}.
\end{align*}
Fix $1<\beta<+\infty$, $0<\gamma<1$. We set $\A$ as
\begin{align*}
\A(x)\coloneqq \sum_{h=1}^{+\infty}h^{-2}x^{\gamma}\chi_{I_{h}^{1}}(x)+ \sum_{h=1}^{+\infty}h^{-2} x^{\beta}\chi_{I_{h}^{2}}(x)
\end{align*}
for every $x\in (0,1)$ and $\A(x)=\A(2-x)$ for every $x\in (1,2)$.
Note that $\|\A\|_\infty\leq 1$, $I_{\Omega,\A}=(0,2)$ and $\A\in BV((0,2))$. Since we defined the function $\A$ by simmetry in the interval $(0,2)$, it is enough to consider his behaviour only in the interval $(0,1)$. Notice that
\begin{align*}
\frac1{\astar(x)}=
\begin{dcases}
h^{2}\left(
\frac{1}{2}\left(\frac{1}{h+1}+\frac1h
\right)
\right)^{-\beta}
 \hskip 0,1cm &\text{if $x\in \left(0,\frac{1}{2}\right)\cap I_{h}^1$}\\
h^{2}x^{-\beta}  &\text{if $x\in \left(0,\frac{1}{2}\right)\cap I_{h}^2$},
 \\
2^{\beta} h^{2} &\text{ if } {{1}\over{2}}\leq x\leq 1,
\end{dcases}
\end{align*}
and so
\begin{align*}
\astar(x)=
\begin{dcases}
h^{-2}\left(
\frac{1}{2}\left(\frac{1}{h+1}+\frac1h
\right)
\right)^{\beta}
 \hskip 0,1cm &\text{if $x\in \left(0,\frac{1}{2}\right)\cap I_{h}^1$},\\
h^{-2}x^{\beta}  &\text{if $x\in \left(0,\frac{1}{2}\right)\cap I_{h}^2$},
 \\
2^{-\beta} h^{-2} &\text{ if } {{1}\over{2}}\leq x\leq 1,
\end{dcases}
\end{align*}
and $\astar(0)=0$.
On the other hand, we set  $u(x)=\frac1{x^3}\in W^{1,1}_{\rm{loc}}((0,1)).$
%
By definition of $\A$ and $u$, we get that $\int_{0}^{1}u(x)\A(x)\de x=+\infty$. Indeed, note that

\begin{align*}
\int_{0}^{1}u(x)\A(x)\de x&=\sum_{h=1}^{\infty}h^{-2}\int_{I_{h}^{1}}x^{\gamma-3}\de x + \sum_{h=1}^{\infty}h^{-2}\int_{I_{h}^{2}}x^{\beta-3}\de x\\
&\sim \sum_{h=1}^{\infty}h^{-\gamma} + \sum_{h=1}^{\infty}h^{-\beta}
\end{align*}
which diverges because $\gamma<1$. On the other hand, we have that

\begin{align*}
\int_{0}^{1}u(x)\astar(x)\de x=\int_{0}^{\frac{1}{4}}u(x)\astar(x)\de x+ \int_{\frac{1}{4}}^{1}u(x)\astar(x)\de x,
\end{align*}
and by definition of $\astar$, $u$, and since the number of intervals of the form $I_{h}^{1},I_{h}^{2}$ contained in $\left(\frac{1}{2},1\right)$ is finite, then the term $\int_{\frac{1}{2}}^{1}u(x)\astar(x)\de x$ is finite. Let us note that

$$
\int_{0}^{\frac{1}{2}}u(x)\astar(x)\de x
=\sum_{h=1}^{\infty}h^{-\beta}
$$
which are convergent because $\beta>1$.  Lastly, let us take a compact set $K\subset (0,1/2)$ such that its interior is a non-empty set. Note that
\begin{align*}
|(\A,Du)|(K)=\int_{K}\vert u^{\prime}\A\vert \de x\sim \sum_{h=1}^{\infty}h^{-\gamma-1} + \sum_{h=1}^{\infty}h^{-\beta-1}
\end{align*}
which is finite because $\beta >1,\ \gamma>0$, and thus we are done.
\end{oss}

\subsection{A weighted Poincar\'e inequality}
The following inequality is a first step into the proof of a Poincar\'e-type inequality in the domain $I_{\Omega,\A}$.
\begin{prop}
Suppose that ${\rm (H1)-(H2)}$ hold true, and $\A\in {\rm BV}_{\rm loc}(I_{\Omega,\A})$. Fix $1\leq i\leq N_{w}$. For all $u\in {\rm{Dom}}_{\A}$, and any $\eta,x$ such that $a_i<\eta\leq x\leq {{a_i+b_i}\over{2}}$ we have
\begin{equation}\label{b1}
|u(x)-u(\eta)|\,\astar(\eta)
\leq  \int_{\eta}^x|u'(y)|\A(y)\,\de y\,;
\end{equation}
\begin{equation}\label{b2}
|u(\eta)|\astar(\eta)\leq |u(x)|\astar(\eta)+ \int_{a_i}^{x}|u'(y)|\A(y)\,\de y\,.
\end{equation}
For every $\eta,x$ such that ${{a_i+b_i}\over{2}}\leq x\leq \eta<b_i$ we have
\begin{equation}\label{b3}
|u(x)-u(\eta)|\,\astar(\eta)
\leq  \int_{x}^\eta|u'(y)| \A(y)\,\de y\,;
\end{equation}
\begin{equation}\label{b4}
|u(\eta)|\astar(\eta)\leq |u(x)|\astar(\eta)+ \int_{x}^{b_i}|u'(y)|\A(y)\,\de y\,.
\end{equation}
\end{prop}
\begin{oss}\label{bb5}
By \eqref{b4} we have $u\astar\in L^\infty(({{a_i+b_i}\over{2}},b_i))$. Indeed, for every $\eta$ such that ${{a_i+b_i}\over{2}}\leq \eta<b_i$
\begin{equation}\label{b5}
|u(\eta)|\astar(\eta)\leq \left|u\left({{a_i+b_i}\over{2}}\right)\right|L_i+ \int_{{{a_i+b_i}\over{2}}}^{b_i}|u'(y)|\A(y)\,\de y<+\infty\,.
\end{equation}
\end{oss}

\begin{proof}
Fix $1\leq i\leq N_{w}$. Let us consider the open set $(\eta,x)\subset (a_i,\frac{a_i+b_i}2)$. We have that
\begin{align*}
\left\vert u(x)-u(\eta)\right\vert &\leq \left\vert \int_{\eta}^{x}u'(y)\de y\right\vert
=\left\vert \int_{\eta}^{x}u'(y)\A(y)\frac1{ \A(y)}\de y\right\vert
\leq  \int_{\eta}^{x}\left\vert u'(y)\right\vert \A(y) \frac{1}{\A(y)}\de y.
\end{align*}
Taking the $\sup$ to $\frac{1}{ \A(y)}$ we obtain 
\begin{align*}
\begin{aligned}
\left\vert u(x)-u(\eta)\right\vert&\leq \int_{\eta}^{x}\left\vert  u'(y)\right\vert \A(y)\de y
\sup_{y\in \left(\eta, \frac{a_{i}+b_{i}}{2}\right)}\frac{1}{  \A(y)}.\\
\end{aligned}
\end{align*}
From the above inequality, we may deduce \eqref{b1}. Further, since
\begin{equation*}
|u(\eta)|\leq|u(x)|+|u(\eta)-u(x)|\,,
\end{equation*}
by \eqref{b1}, \eqref{b2} follows. Similarly, \eqref{b3} and (\ref{b4}) can be obtained.
\end{proof}
\begin{teo}[Poincar\'e type inequality on ${\rm{Dom}}_{\A}$]\label{poincare1}
Suppose that ${\rm (H1)-(H2)}$ hold true, and $\A\in {\rm BV}_{\rm loc}(I_{\Omega,\A})$. Then for every $u\in {\rm{Dom}}_{\A}$
\begin{equation*}\label{poincareformula2}
\sum_{i=1}^{N_{\A}}\medint_{a_i}^{b_i}
\left|u(\eta)-u\left(\frac{a_i+b_i}{2}\right)\right|\astar(\eta)\,d\eta\leq \int_{I_{\Omega,\A}}|u'(y)| \A(y)\,\de y.
\end{equation*}
\end{teo}
\begin{oss}
Let recall that since $u\in {\rm{Dom}}_{\A}$, we need that $\A\in {\rm BV}_{\rm loc}(I_{\Omega,\A})$. 
However, the regularity of $\A$ does not play any role in the proof of Theorem \ref{poincare1}, but it is necessary to define the ambient space ${\rm BV}_{\rm loc}^{\A}$.
Further, let us point out that the hypotheses of Theorem \ref{poincare1}, and \cite[Theorem 2.10]{CCMH} are different. 
Indeed, while the case $1<p<+\infty$ requires a local sommability of $w^{-\frac{1}{p-1}}$, the case $p=1$ requires the local boundedness of $\frac{1}{\A}$. The results of both Theorems are formally analogous, but the auxiliary weights are different, and have different properties. 
Let us also note that we do not assume any local growth condition, as in Theorem \ref{baldipoincarepesata} below, where a weighted Poincaré inequality with a single weight is proved.
\end{oss}

\proof
Fix $1\leq i\leq N_{w}$. In \eqref{b1} we take $x=\frac{a_i+b_i}{2}$, then
\begin{equation*}
\left|u(\eta)-u\left(\frac{a_i+b_i}{2}\right)\right|\astar(\eta)\leq
\int_{a_i}^{\frac{a_i+b_i}{2}}|u'(y)| \A(y)\,\de y.
\end{equation*}
By integrating with respect to $\eta$ we obtain
\begin{equation*}
\int_{a_i}^{\frac{a_i+b_i}{2}}
\left|u(\eta)-u\left(\frac{a_i+b_i}{2}\right)\right|\astar(\eta)\,\de\eta\leq\frac{b_i-a_i}{2}
\int_{a_i}^{\frac{a_i+b_i}{2}}|u'(y)| \A(y)\,\de y.
\end{equation*}
Similarly we have
\begin{equation*}
\int_{\frac{a_i+b_i}{2}}^{b_i}
\left|u(\eta)-u\left(\frac{a_i+b_i}{2}\right)\right|\astar(\eta)\,\de \eta\leq\frac{b_i-a_i}{2}
\int_{\frac{a_i+b_i}{2}}^{b_i}|u'(y)| \A(y)\,\de y.
\end{equation*}
Therefore
\begin{equation*}
\int_{a_i}^{b_i}
\left|u(\eta)-u\left(\frac{a_i+b_i}{2}\right)\right|\astar(\eta)\,\de \eta\leq (b_i-a_i)
\int_{a_i}^{b_i}|u'(y)|\A(y)\,\de y.
\end{equation*}
Hence
\begin{equation*}
\medint_{a_i}^{b_i}
\left|u(\eta)-u\left(\frac{a_i+b_i}{2}\right)\right| \astar(\eta)\,\de\eta\leq
\int_{a_i}^{b_i}|u'(y)| \A(y)\,\de y.
\end{equation*}
The conclusion follows since $u\in {\rm{Dom}}_{\A}$ and so
\begin{equation*}
\sum_{i=1}^{N_{\A}}
\int_{a_i}^{b_i}|u'(y)| \A(y)\,\de y=\int_{I_{\Omega,\A}}|u'(y)| \A(y)\,\de y<+\infty.
\end{equation*}\qed

We also have the following convergence result (see Proposition 9.3 in \cite{COM} for an analogous result).

\begin{prop}\label{proppi}
Suppose that ${\rm (H1)-(H2)}$ hold true, let $\A\in {\rm BV}_{{\rm loc}}(I_{\Omega,\A})$ and let $(u_{k})\subset {\rm BV}_{\rm loc}^{\A}(I_{\Omega,\A})$ be a sequence of functions such that \begin{align}\label{hypo1}
\sup_{k\in \bbN}\left\vert(\A,Du_{k})\right\vert(I_{\Omega,\A})<+\infty, \quad \sup_{k \in \bbN}u_k\left(\frac{a_i+b_i}2\right)<+\infty
\end{align}
for any $i=1,\ldots,N_{\A}$. Then for any interval $K\Subset (a_{i},b_{i})$, with $i=1,\ldots,N_{\A}$,
there exists $u\in L^{1}(K,\astar)\cap W^{1,1}(K)$ 
and a subsequence $(u_{k_{j}})$ such that $u_{k_{j}}\rightarrow u$ in $L^{1}(K, \astar)$.

Moreover, if the sequence $(u_k)_{k \in \bbN}$ is uniformly bounded also in $L^\infty(I_{\Omega,\A})$, i.e.
\begin{align}\label{hypo11111}
\sup_{k\in \bbN}
\left\| u_{k}\right\|_{L^\infty(I_{\Omega,\A})}
+\left\vert(\A,Du_{k})\right\vert(I_{\Omega,\A})<+\infty,
\end{align}
then $u\in {\rm BV}_{\rm loc}^{\A}(I_{\Omega,\A})$.
\end{prop}
\begin{proof}
A first consequence of Theorem \ref{poincare1}  and  \eqref{hypo1} is that 
\begin{align*}
\sup_{k\in \bbN} \norm{u_{k}}_{L^{1}(K,\astar)} <+\infty.
\end{align*}
By \eqref{stima2}, we can find a positive constant $M_{i,K}>0$ such that $\astar(x)>M_{i,K}$ for ${\rm a.e.}$ $x\in K$. Then 
\begin{align*}
M_{i,K}\sup_{k\in\bbN}\int_{K}\vert u_{k}\vert \de x \leq 
\sup_{k\in\bbN}\int_{K}\vert u_{k}\vert \astar \de x <+\infty.
\end{align*}
Moreover, since $\hat w\leq w$
\begin{align*}
M_{i,K}\sup_{k\in\bbN}\int_{K}\vert u_{k}^{\prime}\vert \de x &\leq 
\sup_{k\in\bbN}\int_{K}\vert u_{k}^{\prime}\vert \astar \de x\leq \sup_{k\in\bbN}\int_{I_{\Omega,\A}}\vert u_{k}^{\prime}\vert \astar \de x\leq\\
&\leq\sup_{k\in\bbN}\int_{I_{\Omega,\A}}\vert u_{k}^{\prime}\vert \A \de x\leq \sup_{k\in\bbN} \left\vert(\A,Du_{k})\right\vert(I_{\Omega,\A})<+\infty.
\end{align*}
Then, $(u_{k})_{k}$ is bounded in  $W^{1,1}(K)$. By \cite[Theorem 8.8, Remark 10]{Bre} we can extract a subsequence still denoted $(u_{k})_{k}$, and find $u\in W^{1,1}(K)$ such that 
\begin{align*}
\norm{u_{k}-u}_{L^{1}(K)}\rightarrow 0, \hskip 0,1cm \text{as $k\rightarrow +\infty$.}
\end{align*}
Furthermore, by \eqref{stima1} we can find a constant $L_{i}>0$ such that
\begin{align*}
\int_{K}\vert u_{k}-u\vert \astar \de x \leq L_{i}\int_{K}\vert u_{k}- u\vert\de x\rightarrow 0\hskip0,1cm \text{as $k\rightarrow +\infty$,}  
\end{align*}
and thus we have proved that
\begin{align*}
\norm{u_{k}-u}_{L^{1}(K,\astar)}\rightarrow 0, \hskip 0,1cm \text{as $k\rightarrow +\infty$.}
\end{align*}
Finally, if the sequence $(u_k)_{k \in \bbN}$ is uniformly bounded in $L^\infty(I_{\Omega,\A})$, as in the proof of Proposition 9.3 in \cite{COM} we have that $\|u\|_{L^\infty(K)} \le C$, and this implies that $u \in L^1(K, |D \A|)$. Hence  $u \in BV^{\A}_{\rm{loc}}(I_{\Omega,\A})$.
\end{proof}
\begin{cor} Under the assumptions of Proposition \ref{proppi},
if $N_{\A}<+\infty$, then there exists $u\in L^{1}(I_{\Omega,\A},\astar)$ and a subsequence $(u_{k_{j}})$ such that $u_{k_{j}}\rightarrow u$ in $L^{1}(I_{\Omega,\A}, \astar)$. If \eqref{hypo11111} holds, then $u\in {\rm BV}^{\A}(I_{\Omega,\A})$.
\end{cor}
\begin{proof}
It suffices to use $M=\min\{M_1,\dots,M_{N_{\A}}\}$ and $L=\min\{L_1,\dots,L_{N_{\A}}\}$, instead of $M_i$ and $L_i$, respectively.
\end{proof}


\section{Relaxation for finitely degenerate weights}\label{sec:relax1}
In this section, in addition to hypothesis (H1), (H2) introduced in the previous section, we also make the following assumption on the weight $\A$. 
\begin{itemize}
\item[(H3)] $\A\in W_{\rm loc}^{1,1}(I_{\Omega,\A})$;
\item[(H4)]$1\leq N_{w}<+\infty$.
\end{itemize}

\begin{oss}
We note that, as proven in \cite[Proposition 5.1 (3)]{COM}, under (H3), the space ${\rm BV}^{\A}(I_{\Omega,\A})$ is a Banach space, 
with the norm 
\begin{equation}
||u||_{{\rm BV}^{\A}(I_{\Omega,\A})}:=||u||_{L^1(I_{\Omega,\A},w)}+||u||_{L^1(I_{\Omega,\A}, |Dw|)}+|D(uw)|(I_{\Omega,\A})
\end{equation}
which is not generally the case.
\end{oss}

\subsection{The choice of the ambient space $X$ and the convergence}
Notice that, by the Poincar\'e inequality in Theorem \ref{poincare1}, we have  ${\rm{Dom}}_{\A}\subset L^{1}(\Omega,\hat{\A})$.  In what follows, we set $X=L^{1}(\Omega,\astar)$ and we define the $(\astar,D\A)$-convergence, as follows: 
\begin{defi}\label{ourconv}
We say that a sequence $(u_{n})_{n\in \bbN}\subset {\rm BV}_{{\rm loc}}^{\A}(\Omega)$ $(\astar,D\A)$-converges  to $u\in {\rm BV}_{{\rm loc}}^{\A}(\Omega)$ if
\begin{itemize}
\item[(i)] $u_{n}\weak u$ in $L_{\rm loc}^{1}(I_{\Omega,\A},\astar)$,
\item[(ii)] $u_{n}\weak u$ in $L_{{\rm loc}}^{1}(I_{\Omega,\A},\vert{D}\A\vert)$.
\end{itemize}
\end{defi}
\begin{oss}
This new convergence is a modification of the one introduced in \cite{COM} and guarantees the lower semicontinuity of the pairing functional (see Step 2 in the proof of Theorem \ref{maindegen}).
\end{oss}

\begin{prop}\label{domwBanach}
Suppose that assumptions ${\rm (H1)-(H3)}$ hold true. Then ${\rm{Dom}}_{\A}$ defined as in \eqref{bvspa} is a Banach space endowed with the norm
\begin{align}\label{normavol}
\norm{u}_{{\rm{Dom}}_{\A}}\coloneqq \norm{u}_{L^{1}(I_{\Omega,\A},\astar)}+ \left\vert(\A,Du)\right\vert(I_{\Omega,\A}).
\end{align}
Furthermore, the convergence in \eqref{normavol} implies the $(\astar,D\A)$-convergence.
\end{prop}

\begin{proof}
Notice that ${\rm{Dom}}_{\A}$ is a linear subspace of ${\rm BV}_{\rm loc}^{\A}(I_{\Omega,\A})$, and 
by \cite[Corollary 5.2]{COM} we can endow it with the norm 
\begin{align}\label{nnorm}
\norm{u}_{{\rm BV}^{\A}(I_{\Omega,\A})}\coloneqq \left\vert (\A,Du) \right\vert(I_{\Omega,\A}), \hskip 0,1cm u\in {\rm{Dom}}_{\A}.
\end{align}

\end{proof}

We also have the following compactness result which extends Proposition 9.3 in \cite{COM}. In what follows, we denote by $\LL^{1}$ the unidimensional Lebesgue measure.

\begin{prop}
Suppose that ${\rm (H1)-(H3)}$ hold true, and 
\begin{equation}\label{support}
\LL^1(\Omega\setminus {\rm{supp}}(\astar))=0.
\end{equation}
Let $(u_{k})\subset  {\rm{Dom}}_{\A}
$ be a sequence of functions such that
\begin{align}\label{1novebis}
\sup_{k\in \bbN}\|u_k\|_{L^\infty}+\norm{u_k}_{{\rm{Dom}}_{\A}}
<+\infty.
\end{align}
Then 
there exist $u\in {\rm{Dom}}_{\A}$ 
and a subsequence $(u_{k_{j}})$ such that, possibly up to a further subsequence, 
$u_{k_j} \to u$ in $L^1_{\rm{loc}}(I_{\Omega,\A}, |D \A|)$; so that the sequence $(u_{k_j})_{j \in \bbN}$ locally $(\astar,D\A)$-converges to $u$ in $I_{\Omega,\A}$.
\end{prop}
\begin{proof}
By Proposition \ref{proppi} for any interval $K\Subset (a_{i},b_{i})$, with $i=1,\ldots,N_{\A}$, there exists $u\in L^{1}(K,\astar)\cap W^{1,1}(K)$ and a
 subsequence $(u_{k_{j}})$ such that 
  $u_{k_j}\to u$ in $L^{1}(K, \astar)$ and so
 $u_{k_j}(x) \to u(x)$ for $|\astar| \LL^{1}$-a.e. $x \in K$, and therefore $u_{k_j}(x) \to u(x)$ for $\LL^{1}$-a.e. $x \in {\rm supp}(\astar)\cap K$. 
 Then by \eqref{support} $u_{k_j}(x) \to u(x)$ for $\LL^{1}$-a.e. $x \in K$. Hence, since $|D \A| \ll \LL^{1}$, we get 
\begin{equation*}
u_{k_j}(x) \to u(x) \ \text{ for } \ |D\A|\text{-a.e. } x \in K.
\end{equation*}
Since by \eqref{1novebis}, there exists $C>0$ such that
\begin{equation*}
|u_{k_j} - u| \le  C \in L^1(K, |D\A|),
\end{equation*}
by Lebesgue's Dominated Convergence Theorem we conclude that $u_{k_j} \to u$ in $L^1(\Omega, |D \A|)$.
This implies that $(u_{k_j})_{j \in \bbN}$ $(\astar,D\A)$-converges to $u$ in $K$. 
On the other hand,
by Fatou's Lemma we obtain
\begin{equation*}
\int_{K} |u| \, \de |D\A| = \int_{K} \liminf_{j \to + \infty} |u_{k_j}| \, \de |D \A| \le \liminf_{j \to + \infty} \int_{K} |u_{k_j}| \, \de |D \A| < + \infty.
\end{equation*}
Therefore, we have $u \in L^1_{\rm{loc}}(I_{\Omega,\A}, |D \A|)$, and so $u \in BV^{\A}_{\rm{loc}}(I_{\Omega,\A})$.
\end{proof}

\subsection{Main result}
We then consider
$$
\overline{F}(u):=\inf\{\liminf_{k\to+\infty}F(u_k): u_k\to u {\rm{\ \  w.r.t. \ \ }} (\astar,D\A){\rm{-convergence \ \ }}
\}
$$
where
\begin{equation*}
\!\!\!\!\!\! \!\!\! \!\!\!\!\!\! \! F(u)\coloneqq
\begin{cases}
\displaystyle\int_{\Omega}\vert  u^{\prime}\vert\A\,\de x &\text{ if } u\in {\mathrm AC}(\overline\Omega),\\
+\infty & \text{ if } u\in L^{1}(\Omega,\astar)\setminus  {\mathrm AC}(\overline\Omega),
\end{cases}
\end{equation*}
and let
$$
\mathscr{A}_{\overline{F}}\coloneqq\{u\in L^{1}(\Omega,\astar): \overline{F}(u)<+\infty
\}\,.
$$
Note that for every $u\in {\rm AC}(\overline\Omega)$ we have
$$
\int_\Omega |u^\prime|w\,\de x
=\int_{\Omega}\vert (\A, Du) \vert.
$$
\begin{teo} \label{maindegen} 
Suppose that ${\rm (H1)-(H4)}$ hold true. Then
\begin{equation*}
\mathscr{A}_{\overline{F}}={\rm{Dom}}_{\A}
\end{equation*}
where ${\rm{Dom}}_{\A}$ is defined by \eqref{bvspa} and the following representation holds for the relaxed functional
\begin{equation}\label{Fbar}
\overline{F}(u)=
\begin{cases}
\displaystyle \vert (\A, Du) \vert(I_{\Omega,\A}) &\text{ if } u\in {\rm{Dom}}_{\A},\\
+\infty & \text{ if } u\in L^{1}(\Omega,\astar)\setminus {\rm{Dom}}_{\A} .
\end{cases}
\end{equation}  
\end{teo}
\begin{proof}
Let us denote by $H(u)$ the right-hand side of the above formula \eqref{Fbar}, i.e.
\begin{equation*}\label{H}
H(u)\coloneqq
\begin{cases}
\displaystyle \vert (\A, Du) \vert(I_{\Omega,\A}) &\text{ if } u\in {\rm{Dom}}_{\A},\\
+\infty & \text{ if } u\in L^{1}(\Omega,\astar)\setminus {\rm{Dom}}_{\A} .
\end{cases}
\end{equation*}  
In the following we will prove that $\overline{F}=H$ by showing the two inequalities.

\bigskip\noindent
{\it Step 1}
We first prove that $\overline{F}\le H$. To this end, it is enough to show that
\begin{align}\label{step1}
\overline{F}(u)\leq  \vert (\A, Du) \vert(I_{\Omega,\A})\hskip 0,2cm \text{for all $u\in {\rm{Dom}}_{\A}$}.
\end{align}
Suppose that $AC(\overline\Omega)$ is dense in ${\rm{Dom}}_{\A}$ with respect to \eqref{normavol}.  Then there exists a sequence $(u_{k})$ in $AC(\overline\Omega)$ such that 
\begin{align*}
\lim_{k\rightarrow +\infty}u_{k}=u \hskip 0,1cm \text{in $AC(\overline\Omega)$ with respect to \eqref{normavol}.}
\end{align*}
Then,
\begin{align*}
\overline{F}(u)\leq \lim_{k\rightarrow +\infty}\overline{F}(u_{k})&= \lim_{k\rightarrow +\infty}\vert (\A, Du_{k}) \vert(I_{\Omega,\A})=\vert (\A, Du) \vert(I_{\Omega,\A}),
\end{align*}
which is \eqref{step1}. To complete the proof, 
we now need to show that $AC(\overline\Omega)$ is actually dense in ${\rm{Dom}}_{\A}$ with respect to \eqref{normavol}, 
i.e.,  that for each $u\in  {\rm{Dom}}_{\A} $ there is $u_h\in AC(\overline\Omega)$ such that
\begin{align}\label{approxseqW1}
\begin{aligned}
&\lim_{h\to\infty}u_h=\,u \text{ in } L^{1}(I_{\Omega,\A},\astar
)\hskip 0,1cm \text{and}\\&\vert(\A,Du_{h}) \vert(I_{\Omega,\A}) \rightarrow \vert (\A,Du) \vert(I_{\Omega,\A}) \, \text{as $h\rightarrow +\infty$.}
\end{aligned}
\end{align}
Since $u^\prime\in L^{1}(I_{\Omega,\A},\A)$, we can apply \cite[Theorem 3.45]{Can} to find a sequence of functions $(v_h)_h\subset C^0_c(I_{\Omega,\A})\subset L^{1}(\Omega,\A)$ such that
\begin{equation}\label{nuova11}
\|v_h-u^\prime\|_{L^{1}(I_{\Omega,\A},\A)}=\,\sum_{i=1}^{N_{\A}} \int_{a_i}^{b_i}|v_h-u^\prime|\,\A\,\de x\to 0\text{ as }h\to +\infty\,.
\end{equation}
Let us define, for given $h\in\bbN$, $\tilde{u}^{(i)}_h:(a_i,b_i)\to\R$, $i=1,2,\dots,h$ as
\begin{equation}\label{nuova22}
\tilde{u}^{(i)}_h(x) :=
u\left(\frac{a_i+b_i}{2}\right)-\int_x^{\frac{a_i+b_i}{2}} v_h(y)\,dy\,,\ \quad x\in (a_i,b_i).
\end{equation}
We divide the proof in three cases, according to the structure of the set  $I_{\Omega,\A}$.
\\ 
{\bf 1st case.} Assume that $N_{\A}=1$.  In this case $I_{\Omega,\A}=\,(a_1,b_1)$. 
Let $(\tilde{u}^{(1)}_h)_h$ the sequence defined in \eqref{nuova22} for $i=1$ and, for each $h$, let $u_h=\bar u_h:\,(a,b)\to\R$ defined as
\[
\bar u_h(x):=
\begin{cases}
\tilde{u}^{(1)}_h(a_1)&\text{ if }x\in [a,a_1],\\
\tilde{u}^{(1)}_h(x)&\text{ if }x\in (a_1,b_1),\\
\tilde{u}^{(1)}_h(b_1)&\text{ if }x\in [b_1,b]\,.
\end{cases}
\]
Then it is easy to see that $(\bar u_h)_h\subset AC(\overline\Omega)$. Let us prove that
\begin{equation}\label{CCCf1}
\int_{a}^{b}|\bar u_h-u|\,\astar\,\de x\to 0\text{ as }h\to \infty\,.
\end{equation}
In fact, since $\astar\equiv 0$ in $\Omega\setminus I_{\Omega,\A}$,
$$
\int_{a}^{b}|\bar u_h-u|\,\astar\,\de x=
\int_{a_1}^{b_1}|\bar u_h-u|\,\astar\,\de x.
$$  

By Poincar\'e type inequality \eqref{poincareformula2} with $\tilde{u}_{h}-u$ instead of $u$ and 
since $\tilde{u}_h\left(\frac{a_1+b_1}{2}\right)=\,u\left(\frac{a_1+b_1}{2}\right)$, we have
\begin{equation*}
\begin{split}
\int_{a_1}^{b_1}|\bar u_h-u|\,\astar\,\de x
&\leq \,\int_{I_{\Omega,\A}}|\bar u'_h -u' |\,\A\,\de x
=\vert (D(\bar u_{h}-u),\A) \vert(I_{\Omega,\A})\\
&=\,\int_{I_{\Omega,\A}}|v_h -u' |\,\A\,\de x
\,.
\end{split}
\end{equation*}
Then $\vert (D(\bar u_{h}-u),\A) \vert(I_{\Omega,\A}) \rightarrow 0$, as $h\rightarrow +\infty$.  Hence

\begin{align}\label{25bis}
\Big\vert\vert (D\bar u_{h},\A) \vert(I_{\Omega,\A})- \vert (D u,\A) \vert(I_{\Omega,\A})\Big\vert\leq \vert (D(\bar u_{h}-u),\A) \vert(I_{\Omega,\A}) \rightarrow 0, \hskip 0,1cm \text{as $h\rightarrow +\infty$.}
\end{align}
Moreover, by \eqref{nuova11} and \eqref{25bis}, \eqref{CCCf1} follows.\\
 {\bf 2nd case.}\\
Assume now that $N_w=2$.  In this case $I_{\Omega,\A}=\,(a_1,b_1)\cup (a_2,b_2)$, and assume that $b_1\leq a_2$.
 \\ 
Firstly, we consider the subcase $b_1< a_2$.
 Let $(\tilde{u}^{(i)}_h)_h$ the sequence defined in \eqref{nuova22} for $i=1,2$ and, for each $h$, let $u_h=\bar u_h:\,\Omega\to\R$ defined as
\[
\bar u_h(x):=
\begin{cases}
\tilde{u}^{(1)}_h(a_1)&\text{ if }x\in [a,a_1),\\
\tilde{u}^{(1)}_h(x)&\text{ if }x\in [a_1,b_1),\\
\frac{\tilde{u}^{(2)}_h(a_2)-\tilde{u}^{(1)}_h(b_1)}{a_2-b_1}(x-b_1)+\tilde{u}^{(1)}_h(b_1)&\text{ if }x\in [b_1,a_2),\\
\tilde{u}^{(2)}_h(x)&\text{ if }x\in [a_2,b_2),\\
\tilde{u}^{(2)}_h(b_2)&\text{ if }x\in [b_2,b]\,.
\end{cases}
\]
Notice that $( u_h)_h\subset AC(\overline\Omega)$ and \eqref{approxseqW1} holds.  Indeed, it can be done by repeating the arguments of the 1st case and by observing that $\astar\equiv 0$ in $\Omega\setminus \overline I_{\Omega,\A}$.
 \\ 
 Now, we consider the second subcase $b_1=a_2$. Let $h\in \bbN$ such that 
 $$
 \frac1h<\min\left\{\frac{b_i-a_i}4:i=1,2\right\}.
 $$

 \begin{equation*}
\bar{\A}(x):=
\begin{cases}
\int_{{a_i+b_i}\over{2}}^{x}\astar(y) \,\de y &\text{ if } {a_i}\leq x\leq {{3a_i+b_i}\over{4}},\\
\int_{{3a_i+b_i}\over{4}}^{{a_i+3b_i}\over{4}}\astar(y)\,\de y&\text{ if } {{3a_i+b_i}\over{4}}\leq x\leq {{a_i+3b_i}\over{4}},\\
\int_{x}^{{a_i+b_i}\over{2}}\astar(y)\,\de y &\text{ if } {{a_i+3b_i}\over{4}}\leq x\leq b_i,\\
\ \ \ \ \ \ \ \ \ 0&\text{ if } x\in \Omega\setminus \overline I_{\Omega,\A}\,.
\end{cases}
\end{equation*}
Note that by \eqref{stima1} $\hat \A\in L^\infty((a_i,b_i))$ and so $\bar \A\in L^\infty((a_i,b_i))$.
 
 Let $u_h=\bar u_h:\,\Omega\to\R$ defined as
\[
 \bar u_h(x):=
\begin{cases}
\tilde{u}^{(1)}_h(a_1)&\text{ if }x\in [a,a_1),\\
\tilde{u}^{(1)}_h(x)&\text{ if }x\in [a_1,\frac{a_1+b_1}2),\\
u(x)&\text{ if }x\in [\frac{a_1+b_1}2,b_1-\frac1h),\\
u(x)\frac{\bar{\A}(x)}{\vert \bar{\A}(b_1-\frac1h)\vert}&\text{ if }x\in [b_1-\frac1h,b_1)\\u(x)\frac{\bar{\A}(x)}{\vert \bar{\A}(a_2+\frac1h)\vert}&\text{ if }x\in [a_2,a_2+\frac1h),\\
u(x)&\text{ if }x\in [a_2+\frac1h,\frac{a_2+b_2}2),\\
\tilde{u}^{(2)}_h(x)&\text{ if }x\in [\frac{a_2+b_2}2,b_2),\\
\tilde{u}^{(2)}_h(b_2)&\text{ if }x\in [b_2,b]\,.
\end{cases}
\]

Then $( u_h)_h\subset AC(\overline\Omega)$ and \eqref{approxseqW1} holds.  Indeed,  in order to prove \eqref{approxseqW1}, we now prove that
\begin{equation}\label{CCCfhhhh1}
\int_{\frac{a_1+b_1}2}^{b_1}|\bar u_h-u|\,\astar\,\de x\to 0\text{ as }h\to \infty\,,
\end{equation}
and 
\begin{equation}\label{CCCfhhhhvvv1}
\int_{\frac{a_1+b_1}2}^{b_1}|\bar u_h^\prime|\,\A\,\de x\leq C<+ \infty\,,
\end{equation}
since the proof of the analogous conditions on $(a_2,\frac{a_2+b_2}2)$ are similar.
Indeed, we have

\begin{equation*}
\int_{\frac{a_1+b_1}2}^{b_1}|\bar u_h-u| \,\astar\,\de x=
\int_{b_1-\frac1h}^{b_1}u \left(1-\frac{\bar{\A}(x)}{\vert \bar{\A}(b_1-\frac1h)\vert}\right)\,\astar(x)\,\de x.
\end{equation*}
Notice that $\bar{\A}$ is decreasing in $[{{a_1+3b_1}\over{4}},b_1]$, and by \eqref{stima1}
\begin{equation}
\label{limitata}
0\leq 1-\frac{\bar{\A}(x)}{\vert \bar{\A}(b_1-\frac1h)\vert }=\frac{\vert \bar{\A}(b_1-\frac1h)\vert-\bar{\A}(x)}{\vert \bar{\A}(b_1-\frac1h)\vert }
\leq 
\frac{2L_1}{\vert \bar{\A}(b_1-\frac1h)\vert }
=: \tilde c_h,\qquad x\in \left({b_1-\frac1h},b_1\right).
\end{equation}
Note that $\vert \bar{\A}(b_1-\frac1h)\vert \to \vert \bar{\A}(b_1)\vert \not=0$. Indeed,
$$
\bar{\A}(b_1)=\int_{b_1}^{{a_1+b_1}\over{2}}\astar(y)\,\de y=
\int_{b_1}^{{a_1+b_1}\over{2}}
 \left(\norm{ \A^{-1}}_{L^{\infty}\left(\left(\frac{a_{1}+b_{1}}{2},y\right)\right)}\right)^{-1} 
\,\de y<0.
$$
This implies that
\begin{equation*}
\int_{\frac{a_1+b_1}2}^{b_1}|\bar u_h-u|\,\astar\,\de x\leq
\tilde c_h\int_{b_1-\frac1h}^{b_1}u\,\astar\,\de x
\to 0\text{ as }h\to +\infty\,.
\end{equation*}
This proves \eqref{CCCfhhhh1}.
On the other hand, in order to prove \eqref{CCCfhhhhvvv1} we note that
\[
 \bar u_h'(x):=
\begin{cases}
u'(x)&\text{ if }x\in [\frac{a_1+b_1}2,b_1-\frac1h),\\
\frac{1}{\vert\bar{\A}(b_1-\frac1h)\vert}
\left(u'(x)\bar{\A}(x)+u(x)\bar{\A}'(x)\right)
&\text{ if }x\in [b_1-\frac1h,b_1).
\end{cases}
\]
Therefore
\begin{align*}
\begin{aligned}
\int_{\frac{a_1+b_1}2}^{b_1}|\bar u_h^\prime|\,\A\,\de x&=\int_{\frac{a_1+b_1}2}^{b_1-\frac1h}|u^\prime|\,\A\,\de x+\int_{b_1-\frac1h}^{b_1}\frac{1}{\vert \bar{\A}(b_1-\frac1h)\vert}
\left|u'\bar{\A}+u\bar{\A}'\right|\,\A\,\de x
\\
&\leq\int_{\frac{a_1+b_1}2}^{b_1}|u^\prime|\,\A\,\de x+
\int_{\frac{a_1+b_1}2}^{b_1}\frac{|\bar{\A}|}{\vert \bar{\A}(b_1-\frac1h)\vert}|u'|\A\,\de x+\\&\phantom{form}+\int_{b_1-\frac1h}^{b_1}\frac{1}{\vert \bar{\A}(b_1-\frac1h)\vert}|u|\,|\bar{\A}'|\,\A\,\de x.
\end{aligned}
\end{align*}
Notice that the second integral is finite by \eqref{limitata}. Let us prove that the last integral tends to $0$. Indeed, 
$$
\bar{\A}'
=
-\astar
\quad\text{ a.e. in }\left(b_1-\frac1h,b_1\right)
$$
and, since $u\astar$ is bounded in $(b_1-1/h,b_1)$ (see Remark \ref{bb5}), we obtain 
$$
\int_{b_1-\frac1h}^{b_1}\frac{1}{\vert \bar{\A}(b_1-\frac1h)\vert}|u||\bar{\A}'|
\,\A\,\de x
=\int_{b_1-\frac1h}^{b_1}\frac{1}{\vert \bar{\A}(b_1-\frac1h)\vert}|u||\astar|\A
\,\de x
$$
$$
\leq
C
\frac{1}{\vert \bar{\A}(b_1-\frac1h)\vert}\int_{b_1-\frac1h}^{b_1}\A\,\de x\rightarrow 0\hskip 0,1cm \text{as $h\rightarrow +\infty$.}
$$
{\bf 3rd case.} In the general case $I_{\Omega,\A}=\,\bigcup_{i=1}^{N_{\A}}(a_i,b_i)$ with $b_i\leq a_{i+1}$, for every $i=1,\dots,N_{\A-1}$, it is sufficient to repeat the arguments of the 2nd case for every $i=1,\dots,N_{\A-1}$.
 \bigskip
 
 {\it Step 2} We now prove that $H\le \overline{F}$. To this end, since $$\overline{F}=\sup\{G: G \hbox{\ lower semicontinuous\ and\ } G\le F\},$$ its is enough to show that $H$ is lower semicontinuous and $H\le F$. The last inequality is trivially true, so, 
 we now need to prove the $\liminf$ inequality for $H$. Let $u_{h}\rightarrow u$ with respect to the $(\astar,D\A)$-convergence in $I_{\Omega,\A}$.  Then we have that $u_{h}\weak u$ weakly in $L^{1}(I_{\Omega,\A},\astar)$. 
 By Mazur Lemma there exists a function $f:\bbN\rightarrow \bbN$ and a sequence $\{\alpha_{k,h}:h\leq k\leq f(h)\}$ such that $\alpha_{k,h}\geq 0$, and
\begin{align*}
\sum_{k=h}^{f(h)}\alpha_{k,h}=1
\end{align*}
such that the sequence 
 
 \begin{align*}
 v_{h}\coloneqq \sum_{k=h}^{f(h)}\alpha_{k,h}u_{k}
 \end{align*}
 strongly converges to $u$ in $L^{1}(I_{\Omega,\A},\astar)$ and $L^{1}(I_{\Omega,\A},|Dw|)$. Notice that \eqref{localc} and the definition of $\astar$ imply that $\LL^{1}(I_{\Omega,\A}\backslash \supp(\astar))=0$. Then  $v_{h}(x)\rightarrow u(x)$ for $\LL^{1}$-a.e. $x\in I_{\Omega,\A}$. Since $\A\in W_{\rm loc}^{1,1}(I_{\Omega,\A})\subset L^\infty_{{\rm loc}}(I_{\Omega,\A})$ and \eqref{localc} hold true, then for all compact $K\Subset I_{i}\coloneqq (a_{i},b_{i})$ one gets

 \begin{align*}
 \frac{1}{c_{i,K}}\int_{K}\vert v_{h}-u\vert \de x\leq \int_{K}\vert v_{h}-u\vert \A \de x\leq C\int_{K}\vert v_{h}-u\vert \de x,
 \end{align*}
 for some positive  constant $C$. Then  $v_{h}\rightarrow u$ strongly in $L_{\rm loc}^{1}(I_{\Omega,\A},\A)$, and thus weakly in $L_{{\rm loc}}^{1}(I_{\Omega,\A},\A)$. Hence, since  $v_{h}\rightarrow u$ strongly in $L_{\rm loc}^{1}(I_{\Omega,\A},|D\A|)$, and thus weakly in $L_{{\rm loc}}^{1}(I_{\Omega,\A},|D\A|)$ we conclude that $v_{h}$ $(\A,\frac{1}{2})$-converges to $u$ in the sense of Definition \ref{atwo}.  
Therefore, we may apply Theorem \ref{lscpairing} to conclude the desired lower semicontinuity inequality. Indeed, by \eqref{sci} 
 we get
\begin{align}\label{ultimad}
\liminf_{h\rightarrow +\infty}H(v_{h})&\geq\lim_{h\rightarrow +\infty}\vert (\A, Dv_{h}) \vert(I_{\Omega,\A})
\geq \vert (\A, Du) \vert(I_{\Omega,\A})=H(u).
\end{align}
Now let us prove that \eqref{ultimad} holds true for $u_{h}$. Suppose by contradiction that \eqref{ultimad} is not true for $u_{h}$. By the definition of $\liminf$ we have that 

\begin{align}\label{inft1}
&C_{1}\coloneqq\sup\left\{\inf\left\{\int_{I_{\Omega,\A}}\A\vert u_{m}'\vert\de x: m\geq h \right\};h\in \bbN \right\} < \int_{I_{\Omega,\A}}\A\vert u'\vert\de x,\\\label{inft2}
& C_{2}\coloneqq\sup\left\{\inf\left\{\int_{I_{\Omega,\A}}\A\vert v_{j}'\vert\de x: j\geq h' \right\};h'\in \bbN \right\} \geq\int_{I_{\Omega,\A}}\A\vert u'\vert\de x.
\end{align}
In \eqref{inft2}, we use the definition of $\sup$, so that for all $\varepsilon>0$, there exists $h'\in \bbN$ such that

\begin{align*}
\inf\left\{\int_{I_{\Omega,\A}}\A\vert v_{j}'\vert\de x: j\geq h' \right\}&>C_{2}-\varepsilon\\
&\geq \int_{I_{\Omega,\A}}\A\vert u'\vert\de x-\varepsilon.
\end{align*}

Moreover, by \eqref{inft1}, for all $h\in \bbN$, we get

\begin{align*}
\inf\left\{\int_{I_{\Omega,\A}}\A\vert u_{m}'\vert\de x: m\geq h \right\} < \int_{I_{\Omega,\A}}\A\vert u'\vert\de x.
\end{align*} 
It implies that  $\int_{I_{\Omega,\A}}\A\vert u'\vert\de x$ is not the infimum, so that  there exists $\delta>0$,  such that for each $m\geq h$

\begin{align}\label{inft4}
\int_{I_{\Omega,\A}}\A\vert u_{m}'\vert\de x +\delta < \int_{I_{\Omega,\A}}\A\vert u'\vert\de x,
\end{align}
and the same inequality holds true for all $m'\geq m\geq h$. Now let us choose such $h\geq h'$. Then

\begin{align*}
\int_{I_{\Omega,\A}}\A\vert u'\vert\de x-\varepsilon \leq 
\inf\left\{\int_{I_{\Omega,\A}}\A\vert v_{j}'\vert\de x: j\geq h' \right\}&\leq \inf\left\{\sum_{k=j}^{f(j)}\alpha_{k,j}\int_{I_{\Omega,\A}}\A \vert u_{k}'\vert \de x: j\geq h'\right\}\\
&\hskip -0,8cm\leq \inf\left\{\sum_{k=j}^{f(j)}\alpha_{k,j}\left(\int_{I_{\Omega,\A}}\A\vert u' \vert \de x-\delta\right): j\geq h'\right\}\\
&=\int_{I_{\Omega,\A}}\A\vert u' \vert \de x-\delta.
\end{align*}
Then  $\delta\leq \varepsilon$. Since $\varepsilon>0$ is arbitrary, we get that $\delta\leq 0$, and thus a contradiction because $\delta>0$. 
\end{proof}

\appendix
\section{A pairing beyond $BV$}\label{sec:append}
In this section, we recall the notion of pairing $(\A,Du)$ for functions $u$ that may not be of bounded variation, and we introduce the larger space $BV^w(\Omega),$ where this pairing make sense.
In the definition, we will use a precise representative $u^\frac12$ defined for functions \(u\in L^1_{{\rm loc}}(\Omega)\).
\subsection{Precise representatives}
Firstly, we recall some basic definitions and results about the precise representatives of \(u\in L^1_{{\rm loc}}(\Omega)\) (see \cite[Sections 3.6 and 4.5]{AFP}), where $\Omega\subset \R^{n}$
is an open set. 

We say that a function \(u\in L^1_{{\rm loc}}(\Omega)\) has an {\em approximate limit} 
\(z\in\R\) at
$x\in\Omega$ if
\begin{equation*}
\lim_{r\rightarrow0^{+}}\frac{1}{\LL^{n}\left(  B_r(x)\right)}\int_{B_r\left(  
x\right)
}\left|  u(y)  -z  \right|  \,dy=0\,;
\end{equation*}
and we say that $x$ is a {\em Lebesgue point} of $u$.
The set $S_u\subset\Omega$ of points where this property does not hold is called the
{\em approximate discontinuity set} of $u$, and $\LL^{n}(S_u) = 0$.
For any $x\in \Omega \setminus S_u$ the approximate limit $z$ is uniquely 
determined and is denoted by $z=:\tilde{u}(x)$. 
Let $u = \chi_E$, for a measurable set $E\subset\R^n$; in this case the approximate limit at a point $x\in\R^n$ is also called {\em density} of $E$ at $x$, and it is defined  by
\[
D(E;x) := \lim_{r \to 0^+} \frac{\LL^{n}(E\cap B_r(x))}{\LL^{n}(B_r(x))}
\] 
whenever this limit exists. 

For every Borel function $u : \Omega \to \R$, we denote the {\em sublevel and superlevel sets} of $u$ as
\begin{equation*}
\{u < t\} = \{ x \in \Omega : u(x) < t \} \ \text{ and } \ \{u > t\} = \{ x \in \Omega : u(x) > t \},
\end{equation*}
and we give the definition of the {\em approximate liminf and limsup} at a point $x \in \Omega$  in the following way  
\begin{equation*}
u^-(x) :=
\sup\left\{t\in\overline{\mathbb{R}}\colon
D(\{u < t\};x) = 0\right\},
\quad
u^+(x) :=
\inf\left\{t\in\overline{\mathbb{R}}\colon
D(\{u > t\};x) = 0\right\}
\end{equation*}
(see \cite[Definition 4.28]{AFP}), where $\overline{\mathbb{R}}:=\mathbb{R}\cup\{\pm \infty\}$.
We note that $u^+, u^- : \Omega \to [- \infty, + \infty]$ are Borel functions and the set $S_u^* := \{ x \in \Omega : u^-(x) < u^+(x) \}$ satisfies 
\begin{equation*} 
\LL^{n}(S_u^*) = 0,
\end{equation*} 
so that $u^+(x) = u^-(x)$ for $\LL^{n}$-a.e. $x \in \Omega$ (see \cite[Definition 4.28]{AFP}). 
If $u \in L^{1}_{\rm loc}(\Omega)$, we have
\begin{equation*} 
u^+(x) = u^-(x) = \tilde{u}(x) \ \text{ for all } x \in \Omega \setminus S_u,
\end{equation*} 
and so $S_u^* \subset S_u$.
Therefore, in $\Omega \setminus S_u^*$ we shall write $\tilde{u}(x) := u^+(x) = u^-(x)$, with an abuse of notation.

On the other hand, for every $u \in L^1_{\rm loc}(\Omega)$, we say that \(x\in\Omega\) is an {\em approximate jump point} of \(u\) if
there exist \(a,b\in\R\), \(a\neq b\), and a unit vector \(\nu\in\R^n\) such that 
\begin{equation}\label{f:disc}
\begin{gathered}
\lim_{r \to 0^+} \frac{1}{\LL^{n}(B_r^i(x))}
\int_{B_r^i(x)} |u(y) - a|\, dy = 0,
\\
\lim_{r \to 0^+} \frac{1}{\LL^{n}(B_r^e(x))}
\int_{B_r^e(x)} |u(y) - b|\, dy = 0,
\end{gathered}
\end{equation}
where \(B_r^i(x) := \{y\in B_r(x):\ (y-x)\cdot \nu > 0\}\), and 
\(B_r^e(x) := \{y\in B_r(x):\ (y-x)\cdot \nu < 0\}\).
The triplet \((a,b,\nu)\), uniquely determined by \eqref{f:disc} 
up to a permutation
of \((a,b)\) and a change of sign of \(\nu\),
is denoted by \((\uint(x), \uext(x), \nu_u(x))\).
We observe that
\begin{equation*} 
\umeno(x) =\min\{\uint(x),\uext(x)\} \ \text{ and } \ 
\upiu(x) =\max\{\uint(x),\uext(x)\}
\quad \text{ for all } x\in J_u.
\end{equation*}

Finally, for $u \in L^1_{\rm loc}(\Omega)$ we define the {\em precise representative} of $u$ in $x \in \Omega$ as
\begin{equation*} 
u^{*}(x) := \lim_{r\to0^+}\frac{1}{\LL^{n}\left(  B_r(x)\right)} \int_{B_r(x)}u(y) \, d y,
\end{equation*}
whenever the limit exists. It is then clear that 
\begin{equation}\label{f:pr2}
u^*(x)=
\begin{cases}
\tilde{u}(x) & x\in \Omega \setminus S_u, \\
\displaystyle \frac{u^i(x)+ u^e(x)}{2} & x\in J_u.
\end{cases}
\end{equation} 
A priori, it is not clear whether $u^*$ is well posed in $S_u \setminus J_u$, in general. 
However, 
for $u \in BV_{\rm loc}(\Omega)$, it is well known that we have $\Haus{n-1}(S_u \setminus J_u) = 0$, so that $u^*(x)$ exists for $\Haus{n-1}$-a.e $x \in \Omega$ and, up to a $\Haus{n-1}$-negligible set, is given by \eqref{f:pr2}. 

Finally, for every Borel function we define the {\em representative} $u^{\frac12}: \Omega \to \overline{\R}$ as
\begin{equation}\label{traccia12}
u^{\frac12}(x):= \begin{cases} {\frac12}(\umeno(x)+\upiu(x)) & \text{ if } x \in \Omega \setminus Z_u \\
0 & \text{ if } x \in Z_u  \\
\end{cases} 
\end{equation}
where $Z_u := \{ x \in \Omega:  u^+(x) = + \infty \text{ and } u^-(x) = - \infty \}$
and
\begin{equation*}
u^{\frac12}(x) = \tilde{u}(x) \ \text{ for all } x \in \Omega \setminus S_u^*.
\end{equation*} 

If $u \in L^1_{\rm loc}(\Omega)$, we notice that $u^{\frac12}(x) = \tilde{u}(x)$ for all $x \in \Omega \setminus S_u$
and $u^\frac{1}{2}(x) = u^*(x)$ for all $x \in \Omega \setminus (S_u\setminus J_u)$, but we might have $u^*(x) \neq u^{\frac{1}{2}}(x)$ for some $x \in S_u \setminus J_u$ (see Example in \cite{COM} Sect. 2.2).

\subsection{Pairing in the $n$-dimensional case}
In this subsection, we need to recall a general notion of pairing for divergence measure fields, as introduced in \cite[Section 3]{COM}. 

We define $\cD\MM_{{\rm loc}}^{1}(\Omega)$  as the space of all vector fields $\A\in L_{{\rm loc}}^{1}(\Omega,\R^{n})$ whose divergence $\dive\A$ in the sense of distributions belongs to $\MM_{{\rm loc}}(\Omega)$.

First of all, we need suitable ambient classes of summable functions, which naturally depend on the chosen Borel field $w$. Given $\A\in \cD\MM^1_{{\rm loc}}(\Omega)$, we set
\begin{align*}
&X^{\A}(\Omega)\coloneqq \left\{u  \text{ Borel function}: u\in L^{1}(\Omega, \A), u^{\frac12} \in L^{1}(\Omega, \vert \dive \A\vert)
\right\},\\
&X_{{\rm loc}}^{\A}(\Omega)\coloneqq \left\{u  \text{ Borel function}: u\in L_{{\rm loc}}^{1}(\Omega, \A), u^{\frac12} \in L_{{\rm loc}}^{1}(\Omega, \vert \dive \A\vert)
\right\}.
\end{align*}
We now recall the definition of pairing for functions in $X_{{\rm loc}}^{\A}(\Omega)$.
\begin{defi}
Let $\A\in \cD\MM_{{\rm loc}}^{1}(\Omega)$, and $u\in X_{\rm loc}^{\A}(\Omega)$. We define the pairing between $\A$ and $u$ as the distribution
\begin{align*}
(\A,Du): C_{c}^{\infty}(\Omega)\rightarrow \R
\end{align*}
acting as
\begin{align}\label{def:pairing}
\nm{(\A,Du),\varphi}\coloneqq -\int_{\Omega}u^{\frac12} \varphi\,\de \dive\A- \int_{\Omega} u \nabla\varphi\cdot\A\,\de x \hskip 0,1cm \text{\ \ \ \ \ for $\varphi\in C_{c}^{\infty}(\Omega)$.}
\end{align}
\end{defi}

\subsection{Pairing in the one-dimensional case}
Let us notice that for $n=1$, $\Omega\subset \R$ and we have $\dive\A=D\A$. Furthermore, we have that
\begin{align*}
\cD\MM^{1}_{\rm loc}(\Omega)={\rm BV}_{\rm loc}(\Omega) \subset L^\infty_{\rm loc}(\Omega)
\end{align*}
and 
\begin{align}\label{def:pairing2}
\nm{(\A,Du),\varphi}= -\int_{\Omega}u^{\frac12} \varphi\,\de D\A- \int_{\Omega} u \nabla\varphi\cdot\A\,\de x \hskip 0,1cm \text{\ \ \ \ \ for $\varphi\in C_{c}^{\infty}(\Omega)$.}
\end{align}
Then, recalling that if $u\in {\rm BV}_{\rm loc}(\Omega)$, then $u^*(x)=u^\frac12(x)$ for every $x\in\Omega$ and by Proposition 3.5 (3) in \cite{COM} since ${\rm BV}_{\rm loc}(\Omega)\subset L^\infty_{\rm loc}(\Omega)$
$$
\{ u \in {\rm BV}_{\rm loc}(\Omega) : u^\frac12 \in L^1_{\rm loc}(\Omega, |D \A|) \}={\rm BV}_{\rm loc}(\Omega).
$$
We will also need the following classes of functions which are the analogue of ${\rm BV}$-type functions when working with the pairing.
\begin{defi}
Given  $\A\in {\rm BV}_{\rm loc}(\Omega)$, we define the class
\begin{align*}
\begin{aligned}
&{\rm BV}^{\A}(\Omega)\coloneqq \left\{u\in X^{\A}(\Omega):(\A,Du)\in \MM(\Omega)\right\},\\
&{\rm BV}_{\rm loc}^{\A}(\Omega)\coloneqq \left\{u\in X_{\rm loc}^{\A}(\Omega):(\A,Du)\in \MM_{\rm loc}(\Omega)\right\}.
\end{aligned}
\end{align*}
\end{defi}
\begin{oss}\label{contained1}
By Proposition 3.5 in \cite{COM} since $\A \in L_{\rm loc}^\infty(\Omega)$ , then 
\begin{equation*} 
{\rm BV}_{\rm loc}(\Omega) \subseteq BV_{\rm loc}^{\A}(\Omega).
\end{equation*}
\end{oss}

\begin{oss}
As noted in \cite[Remark 3.4]{COM}, the set ${\rm BV}^{\A}_{\rm loc}(\Omega)$ is not a linear space. This is due to the fact that the pairing is, in fact, a nonlinear operation in the second component, representing a departure from the classical ${\rm BV}$-setup. 
Nevertheless, if $w\in W^{1,1}_{\rm{loc}}(\Omega)$, then ${\rm BV}^{\A}_{\rm loc}(\Omega)$ is a linear space (see Corollary 5.3 in \cite{COM}).
\end{oss}

Since the pairing $(\A,Du)$ is affected by the pointwise value of $u^{\frac12}$, then a suitable notion of convergence involving these representatives is introduced in \cite{COM}.

\begin{defi}\label{atwo}
Let $\A\in {\rm BV}_{\rm loc}(\Omega)$. We say that a sequence $(u_{n})_{n\in \bbN}\subset X_{{\rm loc}}^{\A}(\Omega)$ $(\A,\frac12)$-converges to $u\in X_{{\rm loc}}^{\A}(\Omega)$ if
\begin{itemize}
\item[(i)] $u_{n}\weak u $ in $L_{\rm loc}^{1}(\Omega,\A)$,
\item[(ii)] $u_{n}^{\frac12} \weak u^{\frac12} $ in $L_{{\rm loc}}^{1}(\Omega,\vert {D}\A\vert)$.
\end{itemize}
\end{defi}
When $\A\in {\rm W}^{1,1}_{\rm loc}(\Omega)$, then (ii) is equivalent to
$u_{n} \weak u $ in $L_{{\rm loc}}^{1}(\Omega,\vert {D}\A\vert).$

The following lower semicontinuity of the pairing holds true.

\begin{teo}[{\cite[Theorem 4.3]{COM}}]\label{lscpairing}
Let $\A\in {\rm BV}_{\rm loc}(\Omega)$. Then for every sequence $(u_{n})_{n\in \bbN} \subset X_{{\rm loc}}^{\A}(\Omega)$ and for every $u\in  X_{{\rm loc}}^{\A}(\Omega) $, and such that $(u_{n})_{n}$\, $(\A,\frac12)$-converges to $u$, it holds 
\begin{align*}
\nm{(\A,Du),\varphi}=\lim_{n\rightarrow +\infty}\nm{(\A,Du_{n}),\varphi} \hskip 0,3cm \text{ for all $\varphi\in C_{c}^{1}(\Omega)$}
\end{align*}
in the sense of distributions. Further, if $u,u_{n}\in {\rm BV}_{{\rm loc}}^{\A}(\Omega)$ for all $n\in \bbN$, then
\begin{align}\label{sci}
\left\vert (\A,Du)\right\vert(\Omega)\leq \liminf_{n\rightarrow +\infty}\left\vert (\A,Du_{n})\right\vert(\Omega).
\end{align} 
If
\begin{align*}\sup_{n\in \bbN}\left\vert (\A,Du_{n})\right\vert(\Omega)<+\infty,
\end{align*}
we get 
\begin{align*}
\left\vert (\A,Du_{n})\right\vert(\Omega)\weak \left\vert (\A,Du)\right\vert(\Omega)
\end{align*}
weakly in the sense of measures.
\end{teo}

\section{Weighted ${\rm BV}$-spaces}\label{baldibv}
In this part, for the sake of completeness, we recall the definition of weighted $\BV(\Omega;\A)$-spaces introduced in \cite{Baldi}, where the weight $\A$ belongs to the global Muckenhoupt's $A_{1}\coloneqq A_{1}(\Omega)$.  Suppose that $\Omega$ is an open subset of $\R$, and let $\Omega_{0}$ be a neighborhood of $\overline{\Omega}$.
\begin{defi}\label{first:A1}
 Let $\A\in L_{{\rm loc}}^{1}(\Omega_{0})$, $\A>0$. We say that $\A\in A_{1}$ if there exists a constant $c>0$ such that 
\begin{align}\label{mucke1}
 \A(x)\geq c \displaystyle\fint_{B(x,r)}\A(y)\de y \hskip 0,1cm \text{a.e. in any ball $B(x,r)\subset \Omega_{0}$.}
\end{align}
\end{defi}
In\cite{Baldi}, given $u\in L^{1}(\Omega;\A)$, the weighted total variation of~$u$ with respect to~$\A$  is defined as
\begin{equation*}
TV(u;\A)\coloneqq\sup\bigg\{\int_{\Omega}u \phi' \,\de x:  \phi\in C_{c}^{1}(\Omega;\R),\vert\phi(x) \vert\leq \A(x)  \text{ for all $x\in \Omega$} \bigg\}.
\end{equation*}
We denote $\BV(\Omega;\A)$ the set of all functions $u\in L^{1}(\Omega;\A)$ for which $TV(u;\A)<+\infty$, and we equipp it with the norm

\begin{align*}
\norma{u}_{\BV(\Omega,\A)}\coloneqq \norma{u}_{L^{1}(\Omega;\A)} + TV(u;\A).
\end{align*}
In particular, when $\A\equiv 1$ we recover the usual space~$\BV(\Omega)$. 
For a measurable set $B\subset \Omega$, we then define the perimeter in~$\Omega$ as the weighted total variation of the characteristic function of~$B$, that is, $\PR(B;\A) \coloneqq TV(\chi_{B};\A)$. 
\begin{oss}
Let us recall that in the definition of weighted Sobolev spaces, the weight is usually defined a.e. (almost everywhere) because functions in these spaces have derivatives that, as measures, 
are absolutely continuous with respect to the Lebesgue measure. 
Nevertheless, in the case of weighted $\BV$-spaces, the situation is completely different. 
Indeed, derivatives can be concentrated on sets of null Lebesgue measure. A proper definition of a weighted $\BV$-space requires a pointwise definition of $\A$. 
In fact, requiring that $\A \in A_{1}$ reflects this, as it captures a pointwise definition in each ball $B(x,r)$ for which the inequality \eqref{mucke1} holds.
\end{oss}
In \cite{Baldi}, it is shown that it not necessary to assume that $\A$ is lower semicontinuous to define a weighted Sobolev space. However, in the case where $\A \in A_{1}$, it is possible to show that we can find an auxiliary weight $\A^{\ast}$ that is lower semicontinuous and such that $\BV(\Omega; \A) = \BV(\Omega; \A^{\ast})$.

\begin{lem}[{\cite[Lemma 3.1]{Baldi}}]\label{equivalentbv}
Suppose that $\A\in A_{1}$. The following assertions hold true.
\begin{enumerate}
\item  Let us set $L_{0}(\Omega,\R)$ the set of Lipschitz continuous functions with compact support. Define

\begin{align*}
\A^{\ast}\coloneqq \sup_{\stackrel{\phi\in L_{0}(\Omega,\R)}{\vert\phi \vert\leq \A}} \vert \phi\vert.
\end{align*}
Then $\BV(\Omega; \A) = \BV(\Omega; \A^{\ast})$.
\item  Let us consider the relaxed function $\A^{\ast\ast}$ associated to $\A$, that is,

\begin{align*}
\A^{\ast\ast}\coloneqq \sup\left\{g: \text{$g:\Omega\rightarrow (0,+\infty)$ is lower semicontinuous, and $g\leq \A$}\right\}.
\end{align*}
Then $\A^{\ast\ast}=\A^{\ast}$ in $\Omega$, and $\BV(\Omega; \A) = \BV(\Omega; \A^{\ast})= \BV(\Omega; \A^{\ast\ast})$.
\item $\A^{\ast\ast}\in A_{1}$.
\end{enumerate}
\end{lem}
Let us set 
\begin{align*}
\widetilde{\A}(x)\coloneqq \sup_{r>0}\displaystyle\fint_{B(x,r)} \A(y)\de y.
\end{align*}
Since $\A\in A_{1}$, note  that  $\widetilde{\A}\in A_{1}$ with the same constant $c>0$. Indeed, observe that

\begin{align*}
\displaystyle\fint_{B(x,r)}\widetilde{\A}(y)\de y\leq \frac{1}{c}\displaystyle\fint_{B(x,r)}\A(y)\de y\leq \frac{1}{c}\widetilde{\A}(x).
\end{align*}
Furthermore, since the integral is a continuous operation, then by taking the supremum of continuous functions we obtain a lower semicontinuous function, and $\widetilde{\A}>0$. Hence, in order to obtain suitable density results, it is customary to replace $\A$ with an appropriate lower semicontinuous function when defining weighted $\BV$-spaces.

\begin{defi}
 Let $\A\in A_{1}$, and define $A_{1}^{\ast}$ as
\begin{align*}
A_{1}^{\ast}\coloneqq \left\{\A\in A_{1}: \text{$\A$ is lower semicontinuous, and condition $A_{1}$ is satisfied at any point}\right\}.
\end{align*}
\end{defi}
The following holds true.

\begin{prop}[{\cite[Theorem 3.3]{Baldi}}]
Let $\A\in A_{1}^{\ast}$, and $u\in \BV(\Omega;\A)$. Then there exist a finite Radon measure $\vert Du\vert_{\A}$ and a $\vert Du\vert_{\A}$-measurable function $\sigma\colon \Omega\rightarrow \R$ such that $\vert \sigma(x)\vert=1$ for $\vert Du\vert_{\A}$-almost every~$x\in \Omega$ and such that
\begin{align}\label{divthm}
\int_{\Omega} u(x) \phi'(x)\,\de x=-\int_{\Omega}\frac{\phi(x)\sigma(x)}{\A(x)}\,\de\vert Du\vert_{\A}(x).
\end{align}
The measure~$\vert Du\vert_{\A}$ and the function~$\sigma$ are uniquely determined by \eqref{divthm} and the weighted total variation $TV(u;\A)$ is equal to $\vert Du\vert_{\A}(\Omega)$.
\end{prop}
Note that, using \eqref{divthm}, one can check that $\vert Du\vert_{\A}= \A \vert Du\vert$, so that
\begin{align*}
TV(u;\A)= \int_{\Omega}\A(x) \,\de \vert Du\vert(x).
\end{align*}
Since the functional $TV(\cdot;\A)$ is defined as the supremum of linear continuous functionals in $L^{1}(\Omega;\A)$, it is lower semicontinuous with respect to the $L^{1}(\Omega;\A)$ metric. The following density theorem for weighted {\rm BV} functions holds true.
\begin{teo}[{\cite[Theorem 3.4]{Baldi}}]\label{densbv}
Let $\Omega$ be an open subset of $\R$ with Lipschitz boundary. Suppose that $\A\in {\rm Lip}(\Omega)$, and $\A\in A_{1}$. 
Then for every $u\in {\rm BV}(\Omega;\A)$ there exists a sequence $\{u_{n}\}_{n\in\bbN}\subseteq C_{c}^{\infty}(\R)$ 
such that $u_{n}\rightarrow u$ in $L^{1}(\Omega)$ and $\int_{\Omega}\vert u_{n}' \vert \A \,\de x \rightarrow TV(u;\A)$ as $n\rightarrow \infty$.
\end{teo}
A similar version of this density result can be found in \cite[Proposition 2.4]{Trillos}. In what follows,  we recall a Poincaré inequality proved in \cite[Theorem 4.2]{Baldi}.

\begin{teo}[{\cite[Theorem 4.2]{Baldi}}]\label{baldipoincarepesata}
Let $u\in {\rm BV}(\Omega;\A)$, with $\A\in A_{1}^{\ast}$, and $q>1$. Suppose that the local growth condition

\begin{align}\label{localgrowth}
\left(\frac{\displaystyle\int_{B(x,r)}\A(y)\de y}{\displaystyle\int_{B(x,s)}\A(y)\de y}\right)\leq c\left(\frac{r}{s}\right)^{\frac{q}{q-1}}
\end{align}
holds for any pair of balls $B(x,r)\subset B(x,s)$ in $\R$. Then there exist two positive constants $C_{1}, C_{2}$ such that the following inequalities hold true:
\begin{itemize}
\item
\begin{align*}
\left(\fint_{B}\vert u-u_{B}\vert^{q}\A(y) \de y\right)^{\frac{1}{q}}\leq \frac{rC_{1}}{B}TV(u;\A)(B)
\end{align*}
for all balls $B=B(x,r)\subset \R$, where $u_{B}\coloneqq \displaystyle\fint_{B}u(y)\de y$, and 
\begin{align*}
TV(u;\A)(B)\coloneqq \displaystyle\int_{B}\A(x) \,\de \vert Du\vert(x).
\end{align*}
\item Suppose that 
\begin{align*}
\limsup_{R\rightarrow +\infty}R\left(\displaystyle\int_{B(x,R)}\A(y)\de y\right)^{\frac{1}{q}-1}<+\infty.
\end{align*}
Then 
\begin{align*}\label{Baldi:Poinc2}
\norma{u}_{L^{q}(\Omega;\A)}\leq C_{2}TV(u;\A)(\R).
\end{align*}
\end{itemize}
\end{teo}

\textsc{Acknowledgments.}
The authors are members of  the Istituto Nazionale di Alta Matematica (INdAM), GNAMPA Gruppo Nazionale per l'Analisi Matematica, la Probabilità e le loro Applicazioni, and are partially supported by the INdAM--GNAMPA 2023 Project \textit{Problemi variazionali degeneri e singolari} and the INdAM--GNAMPA 2024 Project \textit{Pairing e div-curl lemma: estensioni a campi debolmente derivabili e diﬀerenziazione non locale}. Part of this work was undertaken while the first and third authors were visiting Sapienza University and SBAI Department in Rome. They would like to thank these institutions for the support and warm hospitality during the visits. 
\\
This study was carried out within the ''2022SLTHCE - Geometric-Analytic Methods for PDEs and Applications (GAMPA)" project – funded by European Union – Next Generation EU  within the PRIN 2022 program (D.D. 104 - 02/02/2022 Ministero dell’Università e della Ricerca). This manuscript reflects only the authors’ views and opinions and the Ministry cannot be considered responsible for them.
\\
\section*{Declarations}
\noindent
{\bf Data Availability} Authors can confirm that all relevant data are included in the article. \\
\vskip 0,1cm
\noindent
{\bf Conflict of interest} The authors confirm that there is no Conflict of interest.

\end{document}